\documentclass[12pt,a4paper]{amsart}

\usepackage{amssymb}
\usepackage[numeric, abbrev, nobysame]{amsrefs}
\usepackage{amscd}

\def\frak{\mathfrak}
\def\Bbb{\mathbb}
\def\Cal{\mathcal}

\let\phi\varphi

\newcommand{\x}{\times}
\renewcommand{\o}{\circ}

\newcommand{\al}{\alpha}
\newcommand{\be}{\beta}

\newcommand{\la}{\lambda}

\newcommand{\ph}{\phi}
\newcommand{\ps}{\psi}

\newcommand{\si}{\sigma}
\newcommand{\ze}{\zeta}
\newcommand{\Ga}{\Gamma}
\newcommand{\La}{\Lambda}
\newcommand{\Ph}{\Phi}

\newcommand{\Om}{\Omega}

\def\Rho{\mbox{\textsf{P}}}

\newcommand{\im}{\operatorname{im}}
\newcommand{\Ric}{\operatorname{Ric}}
\newcommand{\Sc}{\operatorname{Sc}}
\newcommand{\End}{\operatorname{End}}

\newcommand{\id}{\operatorname{id}}

\newcommand{\tr}{\operatorname{tr}}

\renewcommand{\div}{\operatorname{div}}

\newcounter{theorem}
\numberwithin{theorem}{section}
\numberwithin{equation}{section}

\newtheorem{thm}[theorem]{Theorem}
\newtheorem*{thm*}{Theorem \thesubsection}

\newtheorem{prop}[theorem]{Proposition}

\newtheorem*{lemma*}{Lemma \thesubsection}
\newtheorem*{prop*}{Proposition \thesubsection}
\newtheorem*{cor*}{Corollary \thesubsection}

\theoremstyle{definition}
\newtheorem{definition}[theorem]{Definition}
\newtheorem*{definition*}{Definition \thesubsection}

\newtheorem*{example*}{Example \thesubsection}
\theoremstyle{remark}
\newtheorem{remark}[theorem]{Remark}
\newtheorem*{remark*}{Remark \thesubsection}

\def\sideremark#1{\ifvmode\leavevmode\fi\vadjust{\vbox to0pt{\vss
 \hbox to 0pt{\hskip\hsize\hskip1em
 \vbox{\hsize3cm\tiny\raggedright\pretolerance10000
  \noindent #1\hfill}\hss}\vbox to8pt{\vfil}\vss}}}%
                        
\begin{document}

\title{BGG Sequences ---\\ A Riemannian perspective}

\author{Andreas \v Cap}

\address{Faculty of Mathematics\\
  University of Vienna\\
  Oskar--Morgenstern--Platz 1\\
  1090 Wien\\
  Austria}

\date{December 9, 2025}

\email{Andreas.Cap@univie.ac.at}



\begin{abstract}
  BGG resolutions and generalized BGG resolutions from representation theory of
  semisimple Lie algebras have been generalized to sequences of invariant
  differential operators on manifolds endowed with a geometric structure belonging to
  the family of parabolic geometries. Two of these structures, conformal structures
  and projective structures, occur as weakenings of a Riemannian metric respectively
  of a specified torsion-free connection on the tangent bundle. In particular, one
  obtains BGG sequences on open subsets of $\Bbb R^n$ as very special cases of the
  construction. It turned out that several examples of the latter sequences are of
  interest in applied mathematics, since they can be used to construct numerical
  methods to study operators relevant for elasticity theory, numerical relativity and
  related fields.

  This article is intended to provide an intermediate level between BGG sequences for
  parabolic geometries and the case of domains in $\Bbb R^n$. We provide a
  construction of conformal BGG sequences on Riemannian manifolds and of projective
  BGG sequences on manifolds endowed with a volume preserving linear connection on
  their tangent bundle. These constructions do not need any input from parabolic
  geometries. Except from standard differential geometry methods the only deeper
  input comes from representation theory. So one can either view the results as a
  simplified version of the constructions for parabolic geometries in an explicit
  form. Alternatively, one can view them as providing an extension of the simplified
  constructions for domains in $\Bbb R^n$ to general Riemannian manifolds or to
  manifolds endowed with an appropriate connection on the tangent bundle.
  
\end{abstract}

  \maketitle

\pagestyle{myheadings}\markboth{A.\ \v Cap}{Riemannian BGG construction}

\section{Introduction}\label{1}

The origin of BGG sequences lies in pure algebra. For an irreducible representation
$\Bbb V$ of a complex semisimple Lie algebra $\frak g$, I.N.\ Bernstein,
I.M.\ Gelfand and S.I.\ Gelfand constructed in \cite{BGG} a resolution of $\Bbb V$ by
homomorphisms of Verma modules. In \cite{Lepowsky}, this was generalized by
J.\ Lepowsky to a resolution by homomorphisms of generalized Verma modules associated
to a parabolic subalgebra $\frak p\subset\frak g$. These results have a connection to
geometry via a duality relating homomorphisms between certain induced modules to
invariant differential operators acting on sections of homogeneous vector bundles
over a homogeneous space. Via this duality, the homomorphisms showing up in the
resolutions constructed in \cite{BGG} correspond to invariant differential operators
between sections of homogeneous line bundles over the full flag manifold of a complex
Lie group $G$ with Lie algebra $\frak g$. Likewise, in the setting of \cite{Lepowsky}
there is a relation to invariant differential operators acting on sections of
homogeneous vector bundles over the generalized flag manifold $G/P$, that are induced
by irreducible representations of an appropriate parabolic subgroup $P\subset G$.

Via complexification, these results are also related to invariant differential
operators on homogeneous vector bundles over real generalized flag manifolds
$G/P$. For appropriate choices of $G$ and $P$, these generalized flag manifolds are
the homogeneous model for geometric structures, which are of broad interest in
differential geometry. In particular, for the connected component $G:=SO_0(n+1,1)$ of
the identity in the orthogonal group of a Lorentzian inner product in dimension
$n+2$, there is just one parabolic subgroup $P$ up to conjugation, namely the
stabilizer of an isotropic line in $\Bbb R^{n+1,1}$. It then turns out that
$G/P\cong S^n$ with the action identifying $G$ with the group of all orientation
preserving conformal isometries of $S^n$. Thus there is a relation of Lepowsky's
generalized BGG resolutions to conformally invariant differential operators, which was
exploited for example in the work \cite{Eastwood-Rice} of M.G.\ Eastwood and J.W.\
Rice. In the 1970's and 80's it became also clear, that other generalized flag
manifolds are related to interesting geometric structures. In particular, the unitary
group $SU(n+1,1)$ in a similar way leads to the homogeneous model of strictly
pseudo-convex CR structures of hypersurface type, which is the starting point for
several parallel developments in CR geometry and conformal geometry, that turned out
to be very fruitful.

Motivated by this, a general study of geometries with homogeneous model a real or
complex generalized flag manifold was initiated in the 1990's under the name
``parabolic geometries'', see \cite{book} for an introduction. One of the early
successes of this theory was a general version of the BGG construction in the setting
of differential operators in \cite{CSS-BGG}. This construction does not only work for
the homogeneous model (and provide resolutions of locally constant sheaves there) but
for arbitrary curved geometries, where it leads to sequences of differential
operators that are intrinsic to the geometric structure in question. Proving
existence of such operators in the curved setting is a major problem that is solved
by this construction in many cases. The basis for this construction is an equivalent
description of the geometric structures in the family as Cartan geometries. Tools
derived from this description are an essential ingredient, which makes the
construction not easily accessible. While these tools are not needed in what follows,
we have to discuss them briefly to explain the spirit and applications of BGG
sequences and motivate the further developments.

The Cartan description gives rise to a special class of natural vector bundles called
\textit{tractor bundles}. These are rather exotic objects since the action of
morphisms on them depends on higher order jets in a point, but each such bundle comes
with a linear connection canonically associated to the geometry. These so-called
\textit{tractor connections} can be coupled to the exterior derivative to obtain the
twisted de Rham sequence of differential forms with values in the given tractor
bundle. By construction, this is a sequence of differential operators of order one
defined on sections of natural bundles which are intrinsically associated to the
geometry. On the other hand, Lie algebra theory gives rise to algebraic structures
on the bundles of tractor-bundle-valued forms, basically a filtration by smooth
subbundles and a tensorial operation mapping $k$-forms to $(k-1)$-forms. These can
then be used to define natural subquotient bundles that are associated to Lie algebra
homology spaces and are more standard geometric objects. Moreover, the twisted de
Rham sequence can be ``compressed'' to a sequence of higher order operators acting on
sections of these subquotient bundles. By construction, these so-called
\textit{BGG-operators} also are naturally associated to the geometry.

On geometries locally isomorphic to the homogeneous model $G/P$, all tractor
connections are flat. Hence each of the twisted de Rham sequences is a complex and a
fine resolution of the sheaf of local parallel sections of the tractor bundle in
question. Analyzing the construction carefully, one concludes that also the induced
BGG sequence is a complex and computes the same cohomology and thus also is a fine
resolution. For general geometries, the curvature of any tractor connection
equivalently encodes the Cartan curvature, so one does not obtain complexes. But
there still is a close relation between the two sequences which allows one to switch
between a picture of simple operators on complicated bundles and one of complicated
operators on simple bundles. This is a crucial input for many applications of BGG
sequences on non-flat geometries. From the perspective of the current paper, one
should in particular mention applications to the study of conformally compact metrics
and in particular Poincar\'e-Einstein metrics, see e.g.\ \cites{GW,GLW,Mass} and the
analogs of these concepts in projective differential geometry, for example
\cites{Proj-Comp,Proj-Comp2}.

Another line of applications of BGG sequences on curved geometries concerns the first
operators in BGG sequences, which turn out to always define an overdetermined system
of PDE. These contain many important examples like Killing and conformal Killing
operators on all types of tensor bundles. Here one still gets a close relation to
parallel sections of the tractor bundle in question. Modifying the tractor
connections in a way that preserves the relation between the sequences, one can
interpret the construction as \textit{geometric prolongation} which allows to
equivalently rewrite the overdetermined system defined by the first BGG operator in
first order closed form and hence describe its solutions as parallel sections for a
connection. There are several nice applications of this idea, see e.g.\ \cites{BDE,
  EM}, and it was developed systematically in \cite{BCEG}, which also contains
several steps in the direction of the current article.

The second important input for this article came rather unexpected. One example of a
BGG sequence coming from projective differential geometry is a version of the
Riemannian deformation sequence, which (for the flat metric on a domain in
$\Bbb R^3$) is known as the \textit{Calabi complex} or the \textit{fundamental
  complex of linear elasticity}, see \cite{MikE-elasticity}. Studying analytical
properties of this complex with the ultimate goal to develop efficient numerical
methods for its study is of considerable interest in the applied mathematics
community. Finite element methods for differential forms were developed between the
1970s and 1990s and in the first decade of the 2000's, a more complete, efficient
picture for the numerical analysis of (scalar) differential forms was developed
under the name \textit{finite element exterior calculus}, see e.g.\ \cite{AFW1}. This
raised some interest in the geometric BGG construction (that is based on differential
forms), which in turn led to interaction between the communities and progress, see
\cites{AFW2,AFW3}, in that period. Unfortunately, there was not much interaction between
the two communities in the subsequent years and while there was quite a bit of
further activity on BGG-like constructions in the applied community, this was usually
based on ad-hoc constructions and did no use representation theory.

The interaction between the two communities was taken up again around 2020 in
discussions of D.\ Arnold, K.\ Hu, and myself. After a longer period of developing a
common language and exchanging ideas and points of view, this led to the joint
article \cite{weak} of K.\ Hu and myself. In this article, we develop a simplified
version of the BGG construction in the setting of the flat metric and flat connection
on Lipschitz domains in $\Bbb R^n$, which also applies in low (Sobolev)
regularity. This is done in an abstract setting of Hilbert complexes with the BGG
sequences coming from representation theory as a major example. In a followup article
\cite{Poincare}, we have shown that this approach can be used to carry over
constructions of Poincar\'e operators with good analytical properties for scalar
differential forms to the setting of BGG complexes (again in low regularity).

\medskip

This sets the stage for the current paper. I'll present a version of the BGG
construction on Riemannian manifolds which avoids all the technical input on
parabolic geometries and tractors. This can either be viewed as a simplified version
of the geometric construction based on a choice of metric in a conformal class or a
(volume-preserving) connection in a projective class. From that point of view, one
uses the realization of tractor bundles and tractor connections in terms of such a
choice as well as simplifications of the construction caused by the fact that one
does not aim for conformal invariance. Alternatively, it can be viewed as a
generalization of the approach of \cite{weak} to Riemannian manifolds in a smooth
setting. From that point of view, a substantial modification is needed in the
construction of the twisted complex, where curvature terms have to be included in
order to obtain complexes for non-flat metrics. In addition, the operations based on
the Levi-Civita connection do not lead to complexes any more, so alternative
descriptions of the cohomology are needed. On the other hand, the actual BGG
construction is quite close to the one in \cite{weak} (in the setting we consider).

There are some technical difficulties that cannot be avoided, however. On the one
hand, considerable input from representation theory is needed, in particular
Kostant's algebraic Hodge theory developed in \cite{Kostant}. It should be pointed
out that for parabolic geometries, parts of this machinery are automatically present in
the tractor connection, while in the approach of this article, they have to be built
in by hand. So it will not always be obvious why one proceeds in the chosen way. In
order to explicitly describe the bundles showing up in a BGG sequence, one also needs
to geometrically interpret Kostant's description of Lie algebra (co)homology from
\cite{Kostant}. It is possible, however, to simply accept the results in that
direction, a detailed understanding of proofs is not needed (and probably would also
not be very helpful). On the other hand, in order to treat different BGG sequences in
a uniform fashion, a slightly unusual approach to Riemannian geometry is helpful. We
will work quite a lot with bundles induced by representations and natural bundle maps
coming from equivariant maps between these representations. However, in each case of
interest, these can be brought to an explicit form, and we will discuss this in
examples.

\section{Conformal BGG sequences on Riemannian manifolds}\label{2}

\subsection{Background from Riemannian geometry}\label{2.1}
We first recall the relation between representation theory of the orthogonal group
$O(n)$ and natural vector bundles on Riemannian $n$-manifolds. A Riemannian metric
$g$ on an $n$-manifold $M$ can be equivalently described via the orthonormal frame
bundle $p:\Cal OM\to M$, which is a principal fiber bundle with structure group
$O(n)$. Given a representation $\Bbb W$ of $O(n)$ one can form the associated vector
bundle $\Cal OM\x_{O(n)}\Bbb W$, which we will denote by $\Cal WM\to M$. This gives a
functorial relation between representations of $O(n)$ and vector bundles canonically
associated to Riemannian $n$-manifolds, so $O(n)$-equivariant maps between
representations induce vector bundle maps between the corresponding bundles. In
particular, the standard representation of $O(n)$ on $\Bbb R^n$ corresponds to the
tangent bundle $TM$, while the adjoint representation on $\frak o(n)$ corresponds to
the bundle $\frak o(TM)$ of skew symmetric endomorphisms of $TM$. The correspondence
is compatible with all tensorial constructions.

The Levi-Civita connection of $g$ can be equivalently described as a principal
connection on $p:\Cal OM\to M$ which in turn induces a linear connection on each of
the associated bundles $\Cal WM$. We will denote all these connections by $\nabla$
and observe that they are compatible with all tensorial operations, which justifies
the uniform notation. From above, we know that the bundle maps induced by $\frak
o(n)$-equivariant maps between representations are parallel for the Levi-Civita
connection. Let us elaborate on this in a special case that will be important in what
follows. For a representation $\Bbb W$ of $O(n)$ we get the infinitesimal action of
$\frak o(n)$ on $\Bbb W$, which defines a bilinear map $\frak o(n)\x\Bbb W\to\Bbb
W$. Passing to associated bundles, we obtain a bilinear bundle map $\frak o(TM)\x\Cal
WM\to\Cal WM$. We will denote both this map and the induced tensorial operation on
sections by $\bullet$, so for $\Ph\in\Ga(\frak o(TM))$ and $\si\in\Ga(\Cal WM)$ we
get $\Ph\bullet\si\in\Ga(\Cal WM)$.

Denoting all these operations by the symbol $\bullet$ is justified by the fact that
they are compatible with constructions for natural vector bundles in a simple way.
This comes from the compatibility of the infinitesimal representation with
constructions. For example, for $\Bbb W=\Bbb W_1\otimes\Bbb W_2$, we get
$\Cal WM=\Cal W_1M\otimes\Cal W_2M$ and by construction, we get
$\Phi\bullet(\si_1\otimes\si_2)=(\Phi\bullet\si_1)\otimes\si_2+\si_1\otimes(\Phi\bullet\si_2)$,
and so on. The fact that $\bullet$ is parallel for the Levi-Civita connection is
equivalent to the fact that for each vector field $\xi\in\frak X(M)$ and any
$\si\in\Ga(\Cal WM)$, we obtain
\begin{equation}\label{Leibniz}
  \nabla_\xi(\Phi\bullet\si)=(\nabla_\xi\Ph)\bullet\si+\Ph\bullet(\nabla_\xi\si).
\end{equation}

This can be made more explicit for concrete choices. For example, if $\Bbb W=\Bbb R^n$
and $\Cal WM=TM$, then $\Ph\bullet\eta=\Ph(\eta)$ and \eqref{Leibniz} just boils down
to the definition of $\nabla$ on $\frak o(TM)$. If $\Bbb W=\Bbb R^{n*}$, the dual of
the standard representation, then $\Cal WM=T^*M$ and by definition of the dual
representation, we obtain $\Phi\bullet \al=-\al\o\Ph$ for $\al\in\Om^1(M)$, i.e.\
$(\Phi\bullet\al)(\eta)=-\al(\Ph(\eta))$. From this formula, one can easily verify
\eqref{Leibniz} directly. For higher degree forms and, more generally, for
$\binom0k$-tensor fields, one gets
$(\Phi\bullet
t)(\eta_1,\dots,\eta_k)=-\sum_it(\eta_1,\dots,\Phi(\eta_i),\dots,\eta_k)$ and again
\eqref{Leibniz} can be easily verified directly from this formula.

An example for the usefulness of the notation is the description of the curvature of
the Levi-Civita connection on all natural bundles. The Riemann curvature tensor is the
$\binom13$-tensor field $R$ defined by
\begin{equation}\label{R-def}
R(\xi,\eta)(\ze)=\nabla_\xi\nabla_\eta\ze-\nabla_\eta\nabla_\xi\ze-\nabla_{[\xi,\eta]}\ze
\end{equation}
for $\xi,\eta,\ze\in\frak X(M)$. The well known symmetries of $R$ imply that it can
be viewed as a two-form with values in $\frak o(TM)$, i.e.\ as a section of
$\La^2T^*M\otimes\frak o(TM)$. It is well known that for the induced connection on
any natural vector bundle the curvature is described via the natural action of this
two-form. In our notation, this means that for any natural bundle $\Cal WM$, any
$\si\in\Ga(\Cal WM)$ and $\xi,\eta\in\frak X(M)$, we get
\begin{equation}\label{R-on-W}
\nabla_\xi\nabla_\eta\si-\nabla_\eta\nabla_\xi\si-\nabla_{[\xi,\eta]}\si=
R(\xi,\eta)\bullet\si. 
\end{equation}

\subsection{Algebraic setup for conformal BGG's}\label{2.2}
We consider $\Bbb R^{n+2}$, denote the standard basis by $e_0,\dots,e_{n+1}$ and
consider the bilinear form $b$ on $\Bbb R^{n+2}$ defined by
$$
b(x,y):= x_0y_{n+1}+x_{n+1}y_0+\textstyle\sum_{i=1}^nx_iy_i. 
$$
Clearly, this has signature $(1,1)$ on the plane spanned by $e_0$ and $e_{n+1}$ and
equals the standard inner product on the orthogonal subspace spanned by
$e_1,\dots,e_n$, so $b$ is Lorentzian. We denote by $G:=O(b)$ the orthogonal group of
$b$, i.e.\ the group consisting of all $C\in GL(n+2,\Bbb R)$ such that
$b(Cx,Cy)=b(x,y)$ for all $x,y\in\Bbb R^{n+2}$. There is an obvious subgroup in
$O(b)$ consisting of those $C$ for which $Ce_0=e_0$ and $Ce_{n+1}=e_{n+1}$. Any
matrix with this property also maps the subspace spanned by $e_1,\dots,e_n$ to itself
and is orthogonal there, so we can view this subgroup as $O(n)\subset G$.

The Lie algebra $\frak g:=\frak o(b)$ of $G$ consists of all matrices $B$ such that
$0=b(Bx,y)+b(x,By)$ for any $x,y\in\Bbb R^{n+2}$. It is easy to describe this
explicitly, c.f.\ Section 1.6.3 of \cite{book}, as block matrices with blocks of sizes
$1$, $n$, and $1$ of the form
\begin{equation}
  \label{o(b)}
  \begin{pmatrix} a & Z & 0 \\ X & A & -Z^t \\ 0 & -X^t & -a\end{pmatrix} \text{\
    with\ } a\in\Bbb R, X\in\Bbb R^n, Z\in\Bbb R^{n*} \text{\ and\ } A\in\frak o(n).  
\end{equation}
Of course, the subgroup $O(n)$ corresponds to the subalgebra formed by all matrices
with $a=X=Z=0$. The block form can be viewed as defining a vector space decomposition
$\frak g=\frak g_{-1}\oplus\frak g_0\oplus\frak g_1$, with the components spanned by
$X$, $(a,A)$ and $Z$, respectively, so $\frak o(n)\subset\frak g_0$. The Lie bracket
on $\frak g$ is given by the commutator of matrices, which readily implies that this
decomposition is compatible with the Lie bracket in the sense that
$[\frak g_i,\frak g_j]\subset\frak g_{i+j}$ for all $i$, $j$. Here and in what
follows, we agree that $\frak g_\ell=\{0\}$ if $\ell\notin \{-1,0,1\}$. Such a
decomposition is referred to as a \textit{$|1|$--grading} on $\frak g$.

In particular, the Lie bracket on $\frak g$ restricts to a bilinear map
$\frak g_0\x\frak g_{-1}\to\frak g_{-1}$ which extends the standard representation of
$\frak o(n)\subset\frak g_0$ on $\frak g_{-1}=\Bbb R^n$. The remaining elements of
$\frak g_0$ (corresponding to $(a,0)$ with $a\in\Bbb R$) act on $\frak g_{-1}$ as
$-a\id$, which shows that we get an identification of $\frak g_0$ with the conformal
Lie algebra $\frak{co}(n)$ in this way. In the same way, the restriction
$\frak g_0\x\frak g_1\to\frak g_1$ identifies $\frak g_1$ with the dual of the
standard representation of $\frak{co}(n)$.

There is a simple way to realize the $|1|$-grading. Namely, one considers the
element$E\in\frak g_0$ that corresponds to $a=1$ and $A=0$. The adjoint action of $E$
is given by multiplication by $i$ on $\frak g_i$ for any $i=-1,0,1$. The element $E$
also acts diagonalizable on the standard representation $\Bbb R^{n+2}$ of
$\frak o(b)$ and the eigenspace decomposition there corresponds to the block form of
matrices used in \eqref{o(b)} and the eigenvalues are (from top to bottom) $+1$, $0$,
and $-1$. This extends to arbitrary representations and since $\frak g$ is
semisimple, we consider the infinitesimal representation of $\frak o(b)$ on an
irreducible representation $\Bbb V$ of $O(b)$. It then turns out that the eigenvalues
of $E$ form an unbroken string of the form $\la,\la+1,\dots,\la+N$ for some
$\la\in\Bbb Z$ and $N\in\Bbb N$. (This easily follows from the fact that all such
representations can be obtained from the standard representation via tensorial
constructions.)  Viewing this as defining a decomposition
$\Bbb V=\oplus_{i=0}^N\Bbb V_i$ we have the fundamental property that for the
infinitesimal action, we get $\frak g_i\cdot\Bbb V_j\subset\Bbb V_{i+j}$ (with similar
conventions as before).

\subsection{Passing to geometry}\label{2.3}
Fix an irreducible representation $\Bbb V$ of $O(b)$ with the decomposition
$\Bbb V=\oplus_{j=0}^N\Bbb V_j$ as in \S \ref{2.2}. Then we can restrict to the
subgroup $O(n)\subset O(b)$ and one easily concludes that each of the subspace
$\Bbb V_j$ is invariant under the action of $O(n)$. Hence if we form associated
bundles over a Riemannian $n$-manifold $(M,g)$ as in \S \ref{2.1} we get a
decomposition $\Cal VM=\oplus_{j=0}^N\Cal V_jM$. In particular, the components
$\frak g_i$ of $\frak g$ also give rise to associated bundles, which are $TM$ for
$i=-1$, $\frak{co}(TM)$ for $i=0$ and $T^*M$ for $i=1$.  Restricting the
infinitesimal representation of $\frak g$ on $\Bbb V$ (which is $G$-equivariant and
hence $O(n)$-equivariant) to the individual components $\frak g_i$, we get induced
bundles maps as follows:
\begin{itemize}
\item An extension $\bullet:\frak{co}(TM)\x\Cal VM\to\Cal VM$ of the map defined in
  \S \ref{2.1} such that $\frak{co}(TM)\bullet\Cal V_jM\subset\Cal V_jM$ for any
  $j$.
  \item A bilinear bundle map $\bullet:TM\x \Cal VM\to\Cal VM$ such that
    $TM\bullet\Cal V_jM\subset\Cal V_{j-1}M$.
  \item A bilinear bundle map $\bullet:T^*M\x \Cal VM\to\Cal VM$ such that
    $T^*M\bullet\Cal V_jM\subset\Cal V_{j+1}M$. 
  \end{itemize}
  As before, we will also denote by $\bullet$ the induced tensorial operations on
  sections. For a section $s$ of $\Cal VM$, we will denote the component in
  $\Cal V_jM$ by $s_j$ and we will swap between the points of view that $\bullet$ is
  defined on all of $\Cal VM$ or on the individual components. So for a vector field
  $\eta\in\frak X(M)$ and $s\in\Ga(\Cal VM)$ we get $\eta\bullet s\in\Ga(\Cal VM)$
  and $(\eta\bullet s)_j=\eta\bullet s_{j+1}\in\Ga(\Cal V_jM)$. Likewise for
  $\al\in\Om^1(M)$, we get $(\al\bullet s)_j=\al\bullet s_{j-1}\in\Ga(\Cal V_jM)$. As
  we have noted in \S \ref{2.1}, the fact that these bundle maps come from equivariant
  maps between the inducing representation implies that for $\xi\in\frak X(M)$, we get
  \begin{equation}\label{Leibniz2}
    \nabla_\xi(\eta\bullet s)=(\nabla_\xi\eta)\bullet s+\eta\bullet(\nabla_\xi s)
  \end{equation}
  and similarly for $\al\bullet s$ and the components $\eta\bullet s_j$ and
  $\al\bullet s_j$.

  The fact that the infinitesimal representation is compatible with the Lie bracket
  has several important consequences of us. On the one hand, $\frak g_{-1}$ and
  $\frak g_1$ are abelian subalgebras of $\frak g$, which implies that
  \begin{gather}\label{abelian-1}
    \eta_1\bullet(\eta_2\bullet s)=\eta_2\bullet(\eta_1\bullet s)\\
    \label{abelian1} \al_1\bullet(\al_2\bullet s)=\al_2\bullet(\al_1\bullet s). 
  \end{gather}
  On the other hand, the identification of $\frak g_0$ with $\frak{co}(\frak g_{-1})$
  is via the bracket. Thus we see that for $\Ph\in\frak{co}(TM)$ we obtain
  \begin{gather}\label{equiv-T}
    \Ph\bullet(\eta\bullet s)-\eta\bullet(\Ph\bullet s)=(\Ph(\eta))\bullet s\\
    \label{equiv-T*}
    \Ph\bullet(\al\bullet s)-\al\bullet(\Ph\bullet s)=-(\al\o\Ph)\bullet s
  \end{gather}

  \subsection{Examples}\label{2.4}
(1) Let us start with the standard representation $\Bbb V=\Bbb R^{n+2}$, which we know
  decomposes as $\Bbb V_0\oplus\Bbb V_1\oplus\Bbb V_2$, with the summands spanned by
  $e_{n+1}$, $\{e_1,\dots,e_n\}$ and $e_0$, respectively. The subalgebra $\frak
  o(n)$ acts trivially on the first and last summand and via the
  standard representation on $\Bbb V_1$. Hence sections of $\Cal V_0M$ and $\Cal V_2M$
  are just functions, while $\Ga(\Cal V_1M)=\frak X(M)$. Hence we can write
  $s\in\Ga(\Cal VM)$ as a triple $s=(f,\zeta,h)$ for $f,h\in C^\infty(M,\Bbb R)$ and
  $\zeta\in\frak X(M)$ and in the block form of \eqref{o(b)} the $0$-component $f$
  corresponds to the bottom component of the vector. Computing in the block form
  \eqref{o(b)} readily shows that
  \begin{equation}\label{act-on-V}
    \eta\bullet (f,\zeta,h)=(-g(\eta,\zeta),h\eta,0)\qquad \al\bullet
    (f,\zeta,h)=(0,-f\al^\#,\al(\zeta)).
  \end{equation}
  Here we use the usual musical isomorphisms, i.e.\ $\al^\#\in\frak X(M)$ is
  characterized by $g(\al^\#,\eta)=\al(\eta)$ and the inverse isomorphism is denoted
  by $\eta\mapsto\eta^\flat$. Notice that from these formulae, the
  identities \eqref{abelian-1} and \eqref{abelian1} as well as \eqref{equiv-T} and
  \eqref{equiv-T*} for $\Ph\in\Ga(\frak o(TM))$ are easily verified directly.

  One has to be slightly careful with the action of $\frak{co}(TM)$ on $\Bbb V$ as
  obtained in \S \ref{2.3}, though, which is \textit{not} the obvious one: The
  element $E\in\frak g_0$ from \S \ref{2.2} acts by multiplication by $-1$ on $\frak
  g_{-1}$ and hence corresponds to $-\id\in\frak{co}(TM)$. But it acts by
  multiplication by $-1$, $0$ and $1$ on $\Bbb V_0$, $\Bbb V_1$ and $\Bbb V_2$,
  respectively. This corresponds to the well known fact that in the conformal
  picture, the bundles $\Cal V_0M$ and $\Cal V_2M$ are bundles of densities and not
  of functions, while $\Cal V_1M$ is a weighted tangent or cotangent bundle. Using
  this observation, the general versions of \eqref{equiv-T} and \eqref{equiv-T*} can
  be verified directly.

  \smallskip

(2) One can now pass to more general representations and bundles via tensorial
  constructions, but there is a choice of identifications one has to make. For
  example, consider the irreducible representation $\Bbb W:=S^2_0\Bbb V$ with the
  subscript indicating trace-freeness with respect to the Lorentzian inner
  product. The decomposition of $\Bbb V$ easily implies that $\Bbb W=\Bbb
  W_0\oplus\dots\oplus\Bbb W_4$ with $\Bbb W_0=S^2\Bbb V_0$, $\Bbb W_1\cong
  \Bbb V_0\otimes\Bbb V_1$, $\Bbb W_2=(S^2\Bbb V_1\oplus\Bbb V_0\otimes\Bbb V_2)_0$, $\Bbb
  W_3=\Bbb V_1\otimes\Bbb V_2$ and $\Bbb W_4=S^2\Bbb V_2$. As representations of
  $O(n)$, these are just $\Bbb R$ in degree $0$ and $4$, $\Bbb R^n$ in degree $1$ and
  $3$ and $S^2_0\Bbb R^n\oplus\Bbb R$ in degree $2$. Hence a section of $\Cal WM$ can
  be viewed as $(f_1,\zeta_1,(\Ph,f_2),\zeta_2,f_3)$ with $f_i\in C^\infty(M,\Bbb
  R)$, $\ze_i\in\frak X(M)$ and $\Ph$ a trace-free symmetric $\binom20$-tensor
  field. For most components, it is also obvious how to relate them to elements of
  $S^2_0\Cal V$. Denoting symmetric tensor products by $\odot$, the element with only
  non-trivial component $f_1$ can be realized as $(f_1,0,0)\odot (1,0,0)$, and
  likewise the element with only non-trivial component $f_3$ corresponds to
  $(0,0,f_3)\odot (0,0,1)$. The elements with only non-trivial component $\ze_1$ or
  $\ze_2$ can be realized as $(0,\ze_1,0)\odot (1,0,0)$ and $(0,\ze_2,0)\odot
  (0,0,1)$, respectively.

  For the component with only non-trivial entry $\Ph$ things are still easy. We can
  write $\Ph=\sum_\ell \xi_\ell\odot\eta_\ell$ such that $\sum_\ell
  g(\xi_\ell,\eta_\ell)=0$, and this can be realized as $\sum_\ell
  (0,\xi_\ell,0)\odot (0,\eta_\ell,0)$. Finally, a natural element spanning the
  trivial subrepresentation in $\Bbb W_2$ is $-e_0\odot
  e_{n+1}+\tfrac{1}{n}\sum_{i=1}^ne_i\odot e_i$. Thus in terms of a local orthonormal
  frame $s_i$, we can realize the section with only non-zero component $f_2$ as
  $-(f_2,0,0)\odot (0,0,1)+\sum_{i=1}^n\frac{f_2}{n}(0,s_i,0)\odot
  (0,s_i,0)$. Having set up the identification, one can directly apply the usual
  rules for a Lie algebra action on a symmetric product to derive from
  \eqref{act-on-V} the following formulae for $\eta\bullet
  (f_1,\zeta_1,(\Ph,f_2),\zeta_2,f_3)$ and $\al\bullet
  (f_1,\zeta_1,(\Ph,f_2),\zeta_2,f_3)$:
  \begin{equation}\label{act-on-S2V}
    \begin{gathered}
      (-g(\eta,\ze_1),-g(\eta,\Ph)+\tfrac{2-n}nf_2\eta,((\eta\odot\zeta_2)_0,g(\eta,\ze_2)),2f_3\eta,0)\\
      (0,-2f_1\al^\#,(-(\al^\#\odot\ze_1)_0,-\al(\ze_1)),i_\al\Ph+\tfrac{n+2}nf_2\al^\#,\al(\ze_2)). 
    \end{gathered}
  \end{equation}
  Here $i_\al\Ph$ is the contraction between $\al$ and $\Ph$, so for
  $\Ph=\sum_\ell \xi_\ell\odot\eta_\ell$, this is given by
  $\sum_\ell(\al(\xi_\ell)\eta_\ell+\al(\eta_\ell)\xi_\ell)$. Similarly,
  $g(\eta,\Ph)=i_{\eta^\flat}\Ph$, where $\eta^\flat\in\Om^1(M)$ is
  $g(\eta,\_)$. Finally, in terms of a local orthonormal frame $s_i$, the
  trace-free part is
  $(\eta\odot\ze_2)_0=\eta\odot\zeta_2-\frac{1}{n}g(\eta,\ze_2)\sum_is_i\odot s_i$. Again
  the identities \eqref{abelian-1}--\eqref{equiv-T*} can then be verified directly,
  taking into account that the grading element $E$ acts on $\Bbb W_0$,\dots, $\Bbb
  W_4$ by multiplication by $-2$,\dots, $2$.

  \smallskip

  (3) We demonstrate an alternative approach for the example $\Bbb W:=\La^k\Bbb V^*$
  for $k=2,\dots n$.  This is close to the presentation of \cite{Gover-Silhan} (in a
  conformal setting). Here we immediately conclude that we get a decomposition of the
  form $\Bbb W=\Bbb W_0\oplus\Bbb W_1\oplus\Bbb W_2$ with $\Bbb W_0=\La^{k-1}\Bbb
  V_1^*\wedge\Bbb V_2^*$, $\Bbb W_1=\La^k\Bbb V_1^*\oplus (\La^{k-2}\Bbb
  V_1^*\wedge\Bbb V_0^*\wedge\Bbb V_2^*)$ and $\Bbb W_2=\Bbb V_0^*\wedge\La^{k-1}\Bbb
  V_1^*$. As representations of $O(n)$, the first and last are isomorphic to
  $\La^{k-1}\Bbb R^{n*}$, the middle one to $\La^k\Bbb R^{n*}\oplus\La^{k-2}\Bbb
  R^{n*}$. So sections of $\Cal WM$ can be written as triples
  $(\ph_1,(\ph_2,\ph_3),\ph_4)$ with $\ph_3\in\Om^{k-2}(M)$,
  $\ph_1,\ph_4\in\Om^{k-1}(M)$ and $\ph_2\in\Om^k(M)$. We fix the identification by
  requiring that this section maps $k$ sections $(f_i,\zeta_i,h_i)$ of $\Cal VM$ with
  $i=1,\dots k$ to
  \begin{align*}
  \textstyle\sum_i(-1)^{i-1}h_i&\ph_1(\ze_1,\dots,\widehat{\ze_i},\dots,\ze_k)+\ph_2(\ze_1,\dots,\ze_k)\\
  +&\textstyle\sum_{i<j}(-1)^{i+j}(f_ih_j-f_jh_i)\ph_3(\ze_1,\dots,
  \widehat{\ze_i},\dots,\widehat{\ze_j},\dots,\ze_k)\\
  +&\textstyle\sum_i(-1)^{i-1}f_i\ph_4(\ze_1,\dots,\widehat{\ze_i},\dots,\ze_k).
  \end{align*}
  This formula also shows directly how to extract the values of $\ph_1,\dots,\ph_4$
  on vector fields by plugging appropriate sections of $\Cal VM$ into
  $(\ph_1,(\ph_2,\ph_3),\ph_4)$. The usual formulae then show how to convert an
  action on $(\ph_1,(\ph_2,\ph_3),\ph_4)$ into an action on those sections, which
  then by direct computation gives
  \begin{equation}\label{act-on-LkV*}
    \begin{gathered}
      \eta\bullet(\ph_1,(\ph_2,\ph_3),\ph_4)=
      (-i_\eta\ph_2+\eta^\flat\wedge\ph_3, (\eta^\flat\wedge\ph_4,i_\eta\ph_4), 0)\\
     \al\bullet(\ph_1,(\ph_2,\ph_3),\ph_4)=(0,(-\al\wedge\ph_1,i_{\al^\#}\ph_1),i_{\al^\#}\ph_2+\al\wedge\ph_3).  
    \end{gathered}
  \end{equation}

  \medskip

  (4) The last example is $\Bbb W=\frak g$, the adjoint representation. This is
  isomorphic to $\La^2\Bbb V^*$, but we get some special operations here that we will
  need later on. Of course, the decomposition has the form $\Bbb W_0\oplus\Bbb
  W_1\oplus\Bbb W_2$ and is just the shifted version of the $|1|$-grading. Thus the
  most natural identification is $\Cal W_0M=TM$, $\Cal W_1M=\frak{co}(TM)$ and $\Cal
  W_2M=T^*M$, so sections can be viewed as triples $(\ze,\Ph,\ph)$ with $\ze\in\frak
  X(M)$, $\Ph\in\Ga(\frak{co}(TM))$ and $\ph\in\Om^1(M)$. Note that here $\bullet$ is
  induced by the Lie bracket of $\frak g$, which is also used to identify $\frak g_0$
  with $\frak{co}(\frak g_{-1})$. This readily implies that the $\Cal W_0$-component
  of $\eta\bullet (\ze,\Ph,\ph)$ equals $-\Ph(\eta)$ and the $\Cal W_2$-component of
  $\al\bullet (\ze,\Ph,\ph)$ equals $\al\circ\Ph$. So we only have to compute the
  $\Cal W_1M$ components of both operations, and they both come from the bracket
  $\frak g_{-1}\x\frak g_1\to\frak g_0$. We denote the corresponding operation $\frak
  X(M)\x \Om^1(M)\to\Ga(\frak{co}(TM))$ by $\{\ ,\ \}$ which means that
  \begin{equation}\label{act-adj}
    \eta\bullet (\ze,\Ph,\ph)=(-\Ph(\eta),\{\eta,\ph\},0) \qquad \al\bullet
    (\ze,\Ph,\ph)=(0,-\{\ze,\al\},\al\o\Ph).
  \end{equation}
  To compute the operation $\{\ ,\ \}$ explicitly, we just have to take
  $X\in\frak g_{-1}$ and $Z\in\frak g_1$ and compute the map
  $\frak g_{-1}\to\frak g_{-1}$ that sends $Y$ to $[[X,Z],Y]$. This easily implies
  that
  \begin{equation}\label{bracket}
    \{\eta,\ph\}(\xi)=\ph(\eta)\xi+\ph(\xi)\eta-g(\xi,\eta)\ph^\#.
  \end{equation}

  \subsection{The twisted connection}\label{2.5}
  To start the BGG construction, we need more ingredients related to the curvature of
  $g$. For a Riemannian manifold $(M,g)$ of dimension $\geq 3$, let $R$ be the
  Riemann curvature tensor as defined in \eqref{R-def} above. Recall that the
  \textit{Ricci curvature} $\Ric$ is the symmetric $\binom02$-tensor field obtained
  by contraction as
  \begin{equation}\label{Ricci}
    \Ric(\eta,\zeta)=\textstyle\sum_ig(R(\xi_i,\eta)(\ze),\xi_i)
\end{equation}
for a local orthonormal frame $\{\xi_i\}$. The trace $\Sc$ of $\Ric$ then is the
scalar curvature and we denote by $\Ric_0:=\Ric-\frac{1}n\Sc g$ its trace-free
part. (Vanishing of $\Ric_0$ is the definition of Einstein metrics used in Riemannian
geometry.) For some purposes, it is better to use a slight modification of $\Ric$
called the \textit{Schouten tensor} $\Rho$. This can be characterized via
$\Ric=(n-2)\Rho+\tr(\Rho)g$, which easily implies that $\Rho_0=\tfrac1{n-2}\Ric_0$
and $2(n-1)\tr(\Rho)=\Sc$, so $\Rho=\tfrac1{n-2}\Ric_0+\tfrac1{2n(n-1)}\Sc g$. Thus
$\Rho$ contains the same information as $\Ric$. The main advantage of $\Rho$ is that
it shows up in a simple formula for the decomposition of the Riemann curvature into
its tracefree part, called the \textit{Weyl curvature} $W$ and a trace part, see
Section 2.1 of \cite{BEG}. For our purpose, this can be neatly expressed using the
operation $\{\ ,\ \}$ introduced in \S \ref{2.4} (4) above as
\begin{equation}
  \label{Weyl}
  R(\xi,\eta)(\ze)=W(\xi,\eta)(\ze)+\{\xi,\Rho(\eta)\}(\ze)-\{\eta,\Rho(\xi)\}(\ze). 
\end{equation}
This follows directly from the formula in \cite{BEG} using \eqref{bracket}. It is also
proved in a slightly more general setting in Section 1.6.6 of \cite{book}, which
however uses the opposite sign convention for the Schouten tensor. One can actually
view \eqref{Weyl} as the definition of the Weyl curvature, one then has to show that
$W$ has the same symmetries as $R$ but in addition lies in the kernel of all
contractions. The advantage of the form \eqref{Weyl} of the decomposition is that it
immediately extends to the induced connection on any natural vector bundle, one just
has to replace the evaluation on $\ze$ by the action $\bullet$ on a section $s$.

We need a second curvature quantity, the \textit{Cotton--York tensor} $Y$ of
$g$. This is a two-form on $M$ with values in $T^*M$ which is defined by
\begin{equation}\label{Y-def}
  Y(\xi,\eta):=\nabla_\xi \Rho(\eta)-\nabla_\eta\Rho(\xi)-\Rho([\xi,\eta])
\end{equation}
for $\xi,\eta\in\frak X(M)$. Thus $Y$ is the covariant exterior derivative of the
Schouten tensor $\Rho$, see \S\ref{2.6} below. It is well known, c.f.\ \cite{BEG},
that for $n=3$, the Weyl-tensor $W$ always vanishes identically and vanishing of $Y$
is equivalent to conformal flatness of $g$. For $n\geq 4$, it is well known that $W$
vanishes identically if and only if $g$ is conformally flat and that $Y$ can be
obtained as the divergence of $W$, so it also vanishes in the conformally flat case.

\begin{definition}\label{def2.5}
  Consider an irreducible representation $\Bbb V$ of $O(b)$ as in \S \ref{2.2}. Then
  we define the \textit{twisted connection} $\nabla^{\Cal V}$ on $\Cal VM$ by
\begin{equation}
  \label{twisted-def}
  \nabla^{\Cal V}_\xi s:=\nabla_\xi s+\xi\bullet s-\Rho(\xi)\bullet s, 
\end{equation}
where $\nabla$ is the Levi-Civita connection. 
\end{definition}

Since the last two terms in \eqref{twisted-def} are tensorial, this indeed is a
linear connection on $\Cal VM$. Taking the natural decomposition $\Bbb
V=\oplus_{i=0}^N\Bbb V_i$ and the corresponding decomposition $\Cal VM=\oplus\Cal
V_iM$ we see that $\nabla_\xi$ preserves the summands $\Ga(\Cal V_iM)$ but
$\nabla^{\Cal V}_\xi$ does not have this property. Indeed, denoting the components by
subscripts, we get
$$
(\nabla^{\Cal V}_\xi s)_i=\nabla_\xi s_i+\xi\bullet s_{i+1}-\Rho(\xi)\bullet
s_{i-1}. 
$$
Observe that the summand containing the action of $\xi$ corresponds to the operators
called $S$ in \cite{weak}, while the term involving the Schouten tensor is not
present there. The main reason for including this term is the following result.

\begin{thm}\label{thm2.5}
  For $\xi,\eta\in\frak X(M)$ and $s\in\Ga(\Cal VM)$, the curvature $R^{\Cal V}$ of
  $\nabla^{\Cal V}$ is given by
  $$
  R^{\Cal V}(\xi,\eta)(s):=W(\xi,\eta)\bullet s+Y(\xi,\eta)\bullet s,
  $$
  where $W$ and $Y$ are the Weyl curvature and the Cotton--York tensor of $g$,
  respectively. In particular, the connection $\nabla^{\Cal V}$ is flat if and only
  if the metric $g$ is conformally flat.
\end{thm}
\begin{proof}
  We directly use the defining equation
$$
R^{\Cal V}(\xi,\eta)(s)=\nabla^{\Cal V}_\xi\nabla^{\Cal V}_\eta s-\nabla^{\Cal
  V}_\eta\nabla^{\Cal V}_\xi s-\nabla^{\Cal V}_{[\xi,\eta]}s
$$
for the curvature. Taking $s\in\Ga(\Cal VM)$ and the components
$s_i\in\Ga(\Cal V_iM)$, we of course get
$R^{\Cal V}(\xi,\eta)(s)=\sum_i R^{\Cal V}(\xi,\eta)(s_i)$. From the definition in
\eqref{twisted-def} it follows readily that $R^{\Cal V}(\xi,\eta)(s_i)$ may have
non-trivial components of degree $i-2$, $i-1$, $i$, $i+1$, and $i+2$ only. By
definition, the component in degree $i-2$ equals
$\xi\bullet(\eta\bullet s_i)-\eta\bullet(\xi\bullet s_i)$, so this vanishes by
\eqref{abelian-1}. In the same way, \eqref{abelian1} implies vanishing of the
component of degree $i+2$. For the component in degree $i-1$ we immediately get
$$
\nabla_\xi (\eta\bullet s_i)+\xi\bullet(\nabla_\eta s_i)-\nabla_\eta (\xi\bullet
s_i)-\eta\bullet(\nabla_\xi s_i)-[\xi,\eta]\bullet s_i. 
$$
By \eqref{Leibniz2} the first and fourth term add up to $(\nabla_\xi\eta)\bullet s_i$
and likewise the second and third term give $-(\nabla_\eta\xi)\bullet s_i$. But
torsion freeness of $\nabla$ gives $[\xi,\eta]=\nabla_\xi\eta-\nabla_\eta\xi$ so the
whole expression vanishes. The analysis of the component in degree $i+1$ is very
similar, but we have to replace the vector fields acting via $\bullet$ by their image
under $\Rho$. Using the analog of \eqref{Leibniz2} for the action of one-forms and the
defining equation \eqref{Y-def}, one readily concludes that the component in degree
$i+1$ is given by $Y(\xi,\eta)\bullet s_i$.

So it remains to understand the component in degree $i$. Expanding the defining
equation, there are terms containing only Levi-Civita derivatives and in view of
\eqref{R-on-W}, they add up to $R(\xi,\eta)\bullet s_i$. The remaining terms only
come from the double derivatives (and not from the derivative in direction of the Lie
bracket) and they are given by
$$
-\Rho(\xi)\bullet(\eta\bullet s_i)-\xi\bullet(\Rho(\eta)\bullet
s_i)+\Rho(\eta)\bullet(\xi\bullet s_i)+\eta\bullet(\Rho(\xi)\bullet s_i).
$$
Since the bundle $\Cal VM$ is induced by a representation of $\frak g$, the first
term and the last term add up to $\{\eta,\Rho(\xi)\}\bullet s_i$, while the other two
terms add up to $-\{\xi,\Rho(\eta)\}\bullet s_i$. In view of the extension of
\eqref{Weyl} to associated bundles, this completes the proof.
\end{proof}

\subsection{The twisted de Rham sequence}\label{2.6}
The next step in the construction is standard. Any linear connection on a vector
bundle $F$ can be coupled to the exterior derivative to define the so-called
covariant exterior derivative on $F$-valued differential forms. Moreover, the
composition of two instances of this operator can be explicitly described in terms of
the curvature of the initial connection. In particular, starting from a flat
connection on $F$, one obtains a differential complex.

The simplest way to implement this is to view $\ph\in\Om^k(M,F)$, the space of
$F$-valued $k$-forms, as a $k$-linear, alternating map, which associated to $k$
vector fields on $M$ a section of $F$ and is linear over smooth functions in each
entry. Given a linear connection $\nabla$ on $F$, the covariant exterior derivative
$d^{\nabla}:\Om^k(M,F)\to\Om^{k+1}(M,F)$ is characterized by
\begin{equation}\label{dnab-def}
  \begin{aligned}
    (d^{\nabla}\ph)(\xi_0&,\dots,\xi_k)=\textstyle\sum_{i=0}^k(-1)^i\nabla_{\xi_i}
    \ph(\xi_0,\dots,\widehat{\xi_i},\dots,\xi_k)\\
    +&\textstyle\sum_{i<j}(-1)^{i+j}\ph([\xi_i,\xi_j],\xi_0,\dots,\widehat{\xi_i},
    \dots,\widehat{\xi_j},\dots\xi_k), 
  \end{aligned}
\end{equation}
where the hats denote omission. Exactly as for the global formula for the exterior
derivative, one easily verifies directly that this indeed defines an $F$-valued
$k+1$-form.

Applying this to the Levi-Civita connection on any natural bundle $\Cal WM$, we
obtain $d^{\nabla}:\Om^k(M,\Cal WM)\to\Om^{k+1}(M,\Cal WM)$. In the case of a bundle
$\Cal VM$ induced by a representation $\Bbb V$ of $O(b)$ as in \S \ref{2.3}, we can
also form $d^{\nabla^{\Cal V}}:\Om^k(M,\Cal VM)\to\Om^{k+1}(M,\Cal VM)$. It is easy
to describe the relation of this operator to $d^{\nabla}$. Via the decomposition
$\Cal VM=\oplus_{i=0}^N\Cal V_i$ we can decompose any form $\ph\in\Om^k(M,\Cal VM)$
into components $\ph_i\in\Om^k(M,\Cal V_iM)$. Using this, we formulate
\begin{prop}\label{prop2.6}
  For $\ph\in\Om^k(M,\Cal VM)$, we get
  
  (1) $d^{\nabla^{\Cal V}}\ph=d^{\nabla}\ph+\partial\ph+\partial^{\Rho}\ph$, where
  \begin{gather}\label{part}
    \partial\ph(\xi_0,\dots,\xi_k)=\textstyle\sum_{i=0}^k(-1)^i\xi_i\bullet\ph(\xi_0,\dots,\widehat{\xi_i},\dots,\xi_k)\\ \label{part-Rho}
    \partial^{\Rho}\ph(\xi_0,\dots,\xi_k)=\textstyle\sum_{i=0}^k(-1)^i\Rho(\xi_i)\bullet\ph(\xi_0,\dots,\widehat{\xi_i},\dots,\xi_k)
  \end{gather}
  Decomposing into components, we get
  $(d^{\nabla^{\Cal
      V}}\ph)_i=d^{\nabla}\ph_i+\partial\ph_{i+1}+\partial^{\Rho}\ph_{i-1}$.

  (2) $(d^{\nabla^{\Cal V}}(d^{\nabla^{\Cal V}}\ph))_i$ maps $\xi_0,\dots,\xi_{k+1}$
  to
  \begin{align*}
    \textstyle\sum_{j<\ell}(-1)^{j+\ell}(&W(\xi_j,\xi_\ell)\bullet
                                          \ph_i(\xi_0,\dots,\widehat{\xi_j},\dots,\widehat{\xi_\ell}\dots,\xi_k)\\
    +& Y(\xi_j,\xi_\ell)\bullet
       \ph_{i-1}(\xi_0,\dots,\widehat{\xi_j},\dots,\widehat{\xi_\ell}\dots,\xi_k)).
  \end{align*}
  In particular, $(\Om^*(M,\Cal VM),d^{\nabla^{\Cal V}})$ is a complex if and only if
  $g$ is conformally flat.  
\end{prop}
\begin{proof}
  (1) immediately follows from combining the defining equations \eqref{twisted-def}
  and \eqref{dnab-def}. (2) follows by combining general results on the covariant
  exterior derivative (see Section 19.13 of \cite{Michor:topics}) with Theorem
  \ref{thm2.5}.
\end{proof}

The notation $\partial$ is chosen here since these operators are directly induced by
the standard differential in the complex computing the Lie algebra cohomology of the
abelian Lie algebra $\frak g_{-1}$ with coefficients in the representation $\Bbb
V$. This allows for some input from representation theory, which will be very helpful
later.

\begin{remark}\label{rem2.6}
  Directly generalizing parts of \cite{weak}, one can also define a connection on
  $\Cal VM$ via $(\xi,s)\mapsto\nabla_\xi s+\xi\bullet s$. The proof of Theorem
  \ref{thm2.5} then shows that the curvature of this connection is given by
  $R(\xi,\eta)\bullet s$, so loosely speaking the action term does not change the
  curvature. The associated covariant exterior derivative then is explicitly given by
  $d^\nabla\ph+\partial\ph$, so the maps $\partial$ correspond to the $S$-operators
  in \cite{weak}. Applying the covariant exterior derivative to $\ph$ twice, the
  result sends $\xi_0,\dots,\xi_{k+1}$ to
  $$
  \textstyle\sum_{j<\ell}(-1)^{j+\ell}R(\xi_j,\xi_\ell)\bullet
  \ph(\xi_0,\dots,\widehat{\xi_j},\dots,\widehat{\xi_\ell}\dots,\xi_k).
  $$
  So in the case that the metric $g$ is flat, this makes $\Om^*(M,\Cal WM)$ into a
  differential complex, thus providing a direct extension of parts of \cite{weak} to
  this case (in a smooth setting).

  For flat metrics, also each $(\Om^*(M,\Cal V_iM),d^{\nabla})$ is a complex and
  locally, the sum of the cohomologies of these complexes is isomorphic to the
  cohomology of this ``twisted complex''. Locally, one could also derive analogs of
  the $K$-operators used in \cite{weak} via the relation between local parallel
  sections. In contrast to the case of domains in $\Bbb R^n$ as studied in
  \cite{weak}, one cannot expect to derive uniform formulae for those operators,
  though.
\end{remark}

\subsection{The cohomology bundles}\label{2.7}
We have defined the operation $\partial$ on $\Cal VM$-valued differential forms, but
it is evidently tensorial and hence induced by bundle maps
$\La^kT^*M\otimes\Cal VM\to\La^{k+1}T^*M\otimes\Cal VM$ for $k=0,\dots,n-1$, which we
denote by the same symbol. By construction
$\partial(\La^kT^*M\otimes\Cal V_iM)\subset\La^{k+1}T^*M\otimes\Cal V_{i-1}M$ and
Theorem \ref{2.5} readily implies that $\partial\o\partial=0$. Alternatively the
latter fact can be easily verified directly using \eqref{abelian-1}. Hence in each
degree $k$, we have natural subbundles
$\im(\partial)\subset\ker(\partial)\subset\La^kT^*M\otimes\Cal VM$.

\begin{definition}\label{def2.7}
  Consider an irreducible representation $\Bbb V$ of $O(b)$ and the corresponding
  bundle $\Cal VM=\oplus_{i=0}^N\Cal V_iM$. Then for each degree $k=0,\dots,n$,
  consider $\im(\partial)\subset\ker(\partial)\subset\La^kT^*M\otimes\Cal VM$ and
  define the \textit{cohomology bundle} $\Cal H_k^{\Cal
    V}M:=\ker(\partial)/\im(\partial)$.
\end{definition}

The terminology here comes from the fact that $\partial$ is induced by the
differential in the standard complex computing the cohomology of the (abelian) Lie
algebra $\frak g_{-1}$ with coefficients in the representation $\Bbb V$. Hence the
bundle $\Cal H_k^{\Cal V}M$ by construction is the associated bundle corresponding to
the cohomology space $H^k(\frak g_{-1},\Bbb V)$ which naturally carries a
representation of $O(n)$ and of $\frak g_0\cong\frak{co}(n)$. For the following
developments it will not be really necessary to understand what the cohomology
bundles look like, but of course this is needed to deal with examples. It is
important to realize that the cohomology bundles are much smaller that the bundles of
$\Cal VM$-valued differential forms, and the difference gets more significant the
more complicated the representation $\Bbb V$ gets.

In simple cases, the explicit form can be determined by direct computations, but this
is a point where it becomes increasingly important to use information coming from
representation theory. Recall that any finite dimensional representation of $O(n)$ or
$\frak o(n)$ splits as a direct sum of irreducible representations, which do not
contain any non-trivial invariant subspaces. For $G_0\cong CO(n)$ and $\frak
g_0\cong\frak{co}(n)$ the same holds under an additional condition on the
representation which is satisfied in all cases arising in the context of this
article. A key feature of irreducible representations comes from Schur's lemma. An
equivariant map between two irreducible representations is either zero or an
isomorphism and in the complex case, this isomorphism is uniquely determined up to a
nonzero multiple. In particular, this easily implies that on a $\frak
g_0$-irreducible subspace in $\La^k(\frak g_{-1})^*\otimes\Bbb V$ the grading
element $E$ has to act by a scalar multiple of the identity, which implies that it is
contained in $\La^k(\frak g_{-1})^*\otimes\Bbb V_i$ for some index $i$.

Let us illustrate how to use elementary arguments from representation theory in the
example that $\Bbb V$ is the standard representation $\Bbb R^{n+2}$ of $O(b)$. We
know that $\Bbb V=\Bbb V_0\oplus\Bbb V_1\oplus\Bbb V_2$ with $\dim(\Bbb
V_0)=\dim(\Bbb V_2)=1$ and $\dim(\Bbb V_1)=n$. We have also seen above how $\partial$
is compatible with this decomposition. In particular, in degree zero, we get
$H^0(\frak g_{-1},\Bbb V)=\ker(\partial)$ and $\Bbb V_0$ is evidently contained in
there. Next, $\Bbb V_1\cong\Bbb R^n$ is mapped by $\partial$ to $\Bbb
R^{n*}\otimes\Bbb V_0$ and both the source and the target are irreducible. Verifying
directly that this map is non-zero (or using the general information on the structure
of cohomology provided below) we conclude that this component of $\partial$ has to be
an isomorphism. Finally, on $\Bbb V_2$, we obtain the special case $k=0$ of the
sequence 
\begin{equation} \label{sequence}
  \La^k\Bbb R^{n*}\otimes\Bbb V_2\overset{\partial}{\longrightarrow}
\La^{k+1}\Bbb R^{n*}\otimes\Bbb R^n\overset{\partial}{\longrightarrow} \La^{k+2}\Bbb
R^{n*}\otimes\Bbb V_0
\end{equation}
which we have to consider in general for $k=0,\dots n$ to deal with higher
degrees. Now the first and last space in this sequence are always irreducible, while
for most values of $k$, the middle space splits into three irreducible components. In
particular, for $k=0$, the middle space is isomorphic to $L(\Bbb R^n,\Bbb R^n)$ and
these components correspond to multiples of the identity, symmetric trace-free maps
and skew-symmetric maps, respectively. Now one immediately verifies that for $0\leq
k<n$, the first map in the sequence is injective, while the last map is always
surjective. Together with the above, this shows that, in degree $0$, $\partial$ is injective on $\Bbb
V_1\oplus\Bbb V_2$, so $H^0(\frak g_{-1},\Bbb V)=\Bbb V_0$. Next, $\Bbb
R^{n*}\otimes\Bbb V_0\subset\im(\partial)$ and $\partial$ is injective on $\Bbb
R^{n*}\otimes\Bbb V_2$. Hence $H^1(\frak g_{-1},\Bbb V)$ comes only from the middle
space in \eqref{sequence} for $k=0$, and is isomorphic to tracefree symmetric maps
(realized as symmetric maps modulo multiples of the identity). Similarly, one sees
that for $k=1,\dots,n-2$, the cohomology $H^{k+1}(\frak g_{-1},\Bbb V)$ comes from
the middle space in \eqref{sequence} and is isomorphic to the intersection of the
kernels of the complete alternation and the contraction $\La^{k+1}\Bbb
R^{n*}\otimes\Bbb R^n\to\La^k\Bbb R^{n*}$. It is well known that this also is an
irreducible representation of $\frak o(n)$.

\medskip

Detailed information on the cohomology spaces for general representations $\Bbb V$
can be obtained from more advanced representation theory, which requires substantial
background, however. This is based on Kostant's theorem, which was originally proved
in \cite{Kostant}, see also Section 3.3 of \cite{book} for an exposition. For any
complex irreducible representation $\Bbb V$ of $\frak o(b)$ and each degree $k$, the
cohomology $H^k(\frak g_{-1},\Bbb V)$ splits into a direct sum of irreducible
representations. The number of these components is independent of $\Bbb V$ and can be
described as the cardinality of the certain subset in the Weyl group of $\frak o(b)$
(which is a finite group). This subset can be determined algorithmically and knowing
this, the highest weights of the corresponding irreducible components in the
cohomology can be determined algorithmically from the highest weight of $\Bbb V$. The
case of a real representation $\Bbb V$ can then be dealt with via analyzing the
complexification.  In either case, this needs substantial input from representation
theory of semisimple Lie algebras (description of representations by highest weights,
etc.)  and thus is beyond the scope of the current article. In what follows, we will
not discuss in detail how to apply this theory but just state the results that we
need. In particular, the following fundamental facts can be easily deduced from just
knowing the subset of the Weyl group. The relevant subsets are discussed in Examples
4.3.7 and 4.3.8 of \cite{BEastwood}, which uses them for a different propose,
however, and thus does not discuss the relation to Lie algebra cohomology.
\begin{itemize}
\item If $n$ is odd or $n$ is even and $k\neq n/2$, then $H^k(\frak
  g_{-1},\Bbb V)$ is an irreducible representation of $\frak g_0$.
\item If $n$ is even, then $H^{n/2}(\frak g_{-1},\Bbb V)$ decomposes into the sum of
  at most two irreducible representations of $\frak g_0$.
\item $H^0(\frak g_{-1},\Bbb V)\cong\Bbb V_0$
\end{itemize}

Alternatively to the above description as a quotient, one can also realize $\Cal
H_k^{\Cal V}M$ as a subbundle of $\La^kT^*M\otimes\Cal VM$. Indeed, it is well known
that on any finite dimensional representation of $O(n)$, there is a positive definite
inner product that is $O(n)$-invariant. Applying this to each of the representations
$\Bbb V_i$ and $\La^k\Bbb R^{n*}$, we also get an inner product on $\Bbb
V=\oplus_{i=0}^N\Bbb V_i$ and then on $\La^k\Bbb R^{n*}\otimes\Bbb V$ for each
$k=0,\dots,n$. These inner products in turn induce natural positive definite bundle
metrics on each of the bundles $\La^kT^*M\otimes\Cal VM$. Having these at hand, we
can form the subbundle $\Upsilon_k:=\ker(\partial)\cap
\im(\partial)^\perp\subset\La^kT^*M\otimes\Cal VM$ which by construction projects
isomorphically onto $\Cal H_k^{\Cal V}M$. Since the grading element $E$ discussed in
\S \ref{2.2} acts by a scalar on each irreducible representation of $\frak g_0$, we
conclude that for $k\neq \frac{n}2$, the subbundle $\Upsilon_k$ is contained in
$\La^kT^*M\otimes\Cal V_iM$ for some index $i$.

\subsection{The Riemannian BGG construction}\label{2.8}
The BGG construction ``com\-pres\-ses'' the twisted exterior derivative to higher order
operators between the cohomology bundles. We first need appropriate ``inverses'' to
the bundle maps $\partial$. Following \cite{weak} we call these
$T:\La^kT^*M\otimes\Cal VM\to\La^{k-1}T^*M\otimes \Cal VM$ and use the same symbol
for the induced tensorial maps on sections. In terms of the inner products introduced
in the end of \S \ref{2.7} above, $\partial$ induces, in each degree, an isomorphism
$\ker(\partial)^\perp\to \im(\partial)$ and we define $T$ to be the inverse of this
on $\im(\partial)$ and as zero on $\im(\partial)^\perp$. This readily implies the
following properties.
\begin{alignat}{3}\label{T-spaces}
  &\ker(T)=\im(\partial)^\perp \qquad &&\im(T)=\ker(\partial)^\perp
  \qquad &&\Upsilon_k=\ker(\partial)\cap\ker(T) \\
  &\label{T-comp} T\o T=0  &&T\o \partial\o T=T &&\partial\o
  T\o\partial=\partial. 
\end{alignat}
Having these operators at hand, we can proceed similarly as in \cite{weak}. We first
observe that $T\o (d^{\nabla^\Cal V}-\partial)$ maps each $\Om^k(M,\Cal VM)$ to
itself, but the subspace $\Om^k(M,\Cal V_iM)$ is mapped to
$\Om^k(M,\oplus_{\ell\geq i+1}\Cal V_\ell M)$. Hence
$T\o (d^{\nabla^\Cal V}-\partial)$ is a nilpotent operator. On the other hand, we can
compute
$$
T\o d^{\nabla^{\Cal V}}\o T=T\o (d^{\nabla^\Cal V}-\partial)\o T+T\o\partial\o
T=(\id+T\o (d^{\nabla^\Cal V}-\partial))\o T. 
$$
Hence on sections of $\im(T)$, $T\o d^{\nabla^{\Cal V}}$ coincides with $(\id+T\o
(d^{\nabla^\Cal V}-\partial))$ and hence is invertible, with inverse given by
$\sum_{i=0}^\infty (-1)^i(T\o (d^{\nabla^\Cal V}-\partial))^i$ and the sum is
actually finite. Then we define $G:=(\sum_{i=0}^\infty (-1)^i(T\o (d^{\nabla^\Cal
  V}-\partial))^i)\o T$, which obviously implies $\ker(T)\subset\ker(G)$ and
$\im(G)\subset\im(T)$. 

\begin{prop}\label{prop2.8}
For $\al\in\Om^k(M,\Cal VM)$ with $T(\al)=0$, the form $S(\al):=\al-G(d^{\nabla^\Cal
  V}(\al))$ satisfies $T(S(\al))=0$, $S(\al)-\al\in\Ga(\im(T))$ and $T(d^{\nabla^\Cal
  V}(S(\al)))=0$ and is uniquely determined by theses three properties.
\end{prop}
\begin{proof}
The first two properties of $S(\al)$ follow immediately from
$\im(G)\subset\im(T)\subset\ker(T)$. For the last property, we just observe that by
construction $T\o d^{\nabla^{\Cal V}}\o G=T$. If $\ph\in\Om^k(M,\Cal VM)$ also
satisfies the three properties, then $\ph-S(\al)\in\Ga(\im(T))$ and $(T\o
d^{\nabla^{\Cal V}})(\ph-S(\al))=0$. But as verified above, $T\o d^{\nabla^{\Cal V}}$
is invertible on $\Ga(\im(T))$, which implies $\ph=S(\al)$.  
\end{proof}

In particular, we can apply $S$ to $\al\in\Ga(\Upsilon_k)$ to obtain
$S(\al)\in\Ga(\ker(T))\subset\Om^k(M,\Cal VM)$. Since $S(\al)-\al\in\Ga(\im(T))$, we
conclude that the component of $S(\al)$ in $\Ga(\Upsilon_k)$ coincides with $\al$,
whence $S$ is called the \textit{splitting operator}. On the other hand, since
$d^{\nabla^{\Cal V}}(S(\al))\in\Ga(\ker(T))$, we can project it orthogonally to
$\Ga(\Upsilon_{k+1})$ to define $D(\al)$ and obtain a differential operator
$D=D_k:\Ga(\Upsilon_k)\to\Ga(\Upsilon_{k+1})$. These operators are called the
\textit{BGG operators} determined by $\Bbb V$.

\subsection{The conformally flat case}\label{2.9}
If the metric $g$ is conformally flat, then by Proposition \ref{prop2.6}, the twisted
de Rham sequence $(\Om^*(M,\Cal VM),d^{\nabla^{\Cal V}})$ is a complex. This leads
to a nice conceptual understanding of the relation to the BGG sequence.

\begin{thm}\label{thm2.9}
If the twisted de Rham sequence $(\Om^*(M,\Cal VM),d^{\nabla^{\Cal V}})$ is a
complex, then also the BGG sequence $(\Ga(\Upsilon_*),D)$ is a complex and for any
open subset $U\subset M$, the two complexes compute the same cohomology on $U$. In
particular, $(\Ga(\Upsilon_*),D)$ is a fine resolution of the sheaf of local parallel
sections of $\Cal VM$. 
\end{thm}
\begin{proof}
  For $\al\in\Ga(\Upsilon_k)$, consider
  $d^{\nabla^{\Cal V}}(S(\al))\in\Om^{k+1}(M,\Cal VM)$. By Proposition \ref{prop2.8},
  this is a section of $\ker(T)$ and its component in $\Upsilon_{k+1}$ by definition
  is $D(\al)$, so the difference to $D(\al)$ is a section of $\im(T)$. But if
  $d^{\nabla^{\Cal V}}\o d^{\nabla^{\Cal V}}=0$, we conclude that it also lies in the
  kernel of $T\o d^{\nabla^{\Cal V}}$. Hence the uniqueness part of Proposition
  \ref{prop2.8} shows that $d^{\nabla^{\Cal V}}(S(\al))=S(D(\al))$. But then
  $d^{\nabla^{\Cal V}}(S(D(\al)))=0$ and hence $D(D(\al))=0$. This
  shows that $(\Ga(\Upsilon_*),D)$ is a complex and $S$ is chain map to the twisted
  de Rham complex and hence there is an induced map in cohomology.

  Now suppose that $\ph\in\Om^k(M,\Cal VM)$ satisfies $d^{\nabla^{\Cal V}}(\ph)=0$
  and consider $\tilde\ph:=\ph-d^{\nabla^{\Cal V}}(G(\ph))$. This is cohomologous to
  $\ph$ and we know that $T(d^{\nabla^{\Cal V}}(G(\ph)))=T(\ph)$, so $\tilde\ph$ is a
  section of $\ker(T)$. Denoting by $\al$ the component of $\tilde\ph$ in
  $\Ga(\Upsilon_k)$, we immediately conclude that $\tilde\ph$ satisfies the
  properties from Proposition \ref{prop2.8} and hence $\tilde\ph=S(\al)$. Hence the
  map in cohomology is surjective. On the other hand, suppose that
  $\al\in\Ga(\Upsilon_k)$ has the property that $S(\al)=d^{\nabla^{\Cal V}}(\psi)$
  for some $\psi\in\Om^{k-1}(M,\Cal VM)$. Forming $\tilde\psi$ as above and denoting
  by $\be$ its component in $\Ga(\Upsilon_{k-1})$ we get that $\tilde\ps=S(\be)$ and
  $S(\al)=d^{\nabla^{\Cal V}}(S(\be))$ and hence $\al=D(\be)$. This shows injectivity
  of the map in cohomology and hence completes the proof of the first part. The fact
  that one obtains a fine resolution then immediate follows from the corresponding
  fact for the twisted de Rham sequence proved in Proposition \ref{prop2.6}.
\end{proof}

This generalizes the setting of \cite{weak} for smooth sections to conformally flat
Riemannian manifolds. The main difference is that we can still consider the
individual ``rows'' $\Om^*(M,\Cal V_iM)$ of the twisted complex with $d^\nabla$
acting on them, but they are not complexes any more. So for further applications, say
in the direction of \cite{Poincare}, one would have to start from the full twisted
complex, which comes from a flat connection.

\subsection{The conformally non-flat case}\label{2.10}
Without the assumption on conformal flatness, one does not obtain complexes and the
connection between the twisted sequence and the BGG sequence is less tight. Most
applications so far were to the study of the first operator in the BGG sequence, but
there certainly is very interesting potential in analyzing the rest of the
sequence. The general facts are collected in the following result.

\begin{thm}\label{thm2.10}
  Fix a Riemannian manifold $(M,g)$ and an irreducible representation $\Bbb V$ of
  $O(b)$ and consider the corresponding bundle $\Cal VM=\oplus_i \Cal V_iM$ and the
  connection $\nabla^{\Cal V}$.

  (1) Suppose that $s\in\Ga(\Cal VM)$ satisfies $\nabla^{\Cal V}s=0$. Then for the
  component $s_0\in\Ga(\Cal V_0M)$, we get $s=S(s_0)$ and $D(s_0)=0$. Hence mapping
  to the $\Cal V_0M$-component identifies the space of parallel sections of $\Cal VM$
  with a linear subspace of the kernel of the first BGG operator.

  (2) For $k>0$, projection to $\Ga(\Cal H_k^{\Cal V}M)$ identifies
  $\ker(T)\cap\ker(d^{\nabla^{\Cal V}})$ with a linear subspace of the kernel of
  $D:\Ga(\Cal H_k^{\Cal V}M)\to \Ga(\Cal H_{k+1}^{\Cal V}M)$. 
\end{thm}
\begin{proof}
The basic argument is the same for both parts. Suppose that $T(\ph)=d^{\Cal V}\ph=0$
and let $\al$ denote the projection of $\ph$ to $\Ga(\Upsilon_k)$. Then
$\ph-\al\in\Ga(\im(T))$ and hence Proposition \ref{prop2.6} implies that $\ph=S(\al)$,
which in turn implies that $D(\al)=0$. On the other hand, it also shows that the
projection to the component $\Ga(\Upsilon_k)$ is injective, which completes the argument
for (2). In degree zero $T(s)=0$ is satisfied automatically and $d^{\nabla^{\Cal
    V}}=\nabla^{\Cal V}$ so (1) follows, too.
\end{proof}

Following \cite{Leitner}, it has become common in the setting of parabolic geometries
to call sections $\si\in\Ga(\Cal V_0M)$, for which $S(\si)\in\Ga(\Cal VM)$ is
parallel for $\nabla^{\Cal V}$, \textit{normal solutions} of the first BGG
operator. In several cases, these can be explicitly characterized by interesting
(tensorial or differential) conditions. For example, starting with the adjoint
representation $\frak o(b)$, the first BGG operator is the conformal Killing operator
on vector fields, so its kernel consists of all conformal Killing fields on
$M$. Normal conformal Killing field then turn out to be exactly those, which in
addition insert trivially into the Weyl curvature and into the Cotton-York tensor,
see \cite{deformations}.

The reason why we have singled out the case $k=0$ in Theorem \ref{thm2.10} is because
the result in degree zero turns out to be significantly stronger. Indeed, it was
shown in \cite{BCEG} (in a more general context) that one can modify the connection
$\nabla^{\Cal V}$ in such a way that projection to the component in $\Cal V_0M$
induces a bijection between parallel sections of this new connection and the kernel
of the first BGG operator. So in particular, the kernel of the first BGG operator is
of dimension $\leq\dim(\Bbb V)$ (which is an interesting result in its own
right). This actually extends to operators with the same principal part as the first
BGG operator. While the construction in \cite{BCEG} does not provide a natural
construction of these so-called ``prolongation connections'', it turned out that
there also is an invariant construction in the general setting of parabolic
geometries, see \cite{prolongation}.

\subsection{Example}\label{2.11} 
Let us discuss the case $\Bbb V=\Bbb R^{n+2}$ of the standard representation in
detail. From \S \ref{2.7} we know that $H^0(\frak g_{-1},\Bbb V)=\Bbb V_0\cong\Bbb
R$, $H^1(\frak g_{-1},\Bbb V)\cong S^2_0\Bbb R^{n*}\subset \frak g_{-1}^*\otimes\Bbb
V_1$, and that for $2\leq k<n$, $H^k(\frak g_{-1},\Bbb V)\subset\La^k\frak
g_{-1}^*\otimes\Bbb V_1$. To compute the splitting operator in degree zero, we use
the notation from Example \ref{2.4} (1), so we write $s\in\Ga(\Cal VM)$ as
$(f,\eta,h)$ with $f,h\in C^\infty(M,\Bbb R)$ and $\eta\in\frak X(M)$. By definition,
\begin{equation}\label{nab-std}
\nabla^{\Cal
  V}_\xi(f,\eta,h)=(df(\xi)-g(\xi,\eta),\nabla_\xi\eta+h\xi+f\Rho(\xi),dh(\xi)-\Rho(\xi,\eta)).
\end{equation}
By Proposition \ref{prop2.8}, we can obtain the formula for $S(f)$ by choosing $\eta$
and $h$ in such a way that $T(\nabla^{\Cal V}_\xi(f,\eta,h))=0$. The discussion in \S
\ref{2.7} shows that this means that the first component in the right hand side of
\eqref{nab-std} has to vanish, while the second component has to be trace-free. The
first condition is equivalent to $\eta=df^\#$. Inserting this into the second
component (with varying $\xi$) and taking the trace, we obtain $0=\Delta
f+nh+f\tr(\Rho)$ which implies $h=-\tfrac{1}{n}(\Delta+\tr(\Rho))f$. Inserting these
into \eqref{nab-std} the middle component becomes the trace-free part of
$\nabla^2f+\Rho f$ and since this is automatically symmetric, it already lies in
$\Upsilon_1$, so we have obtained the formula for the first BGG operator in this
case. Thus we obtain
\begin{equation}\label{std-ops}
S(f)=(f,df^\#,-\tfrac{1}{n}(\Delta+\tr(\Rho))f) \qquad D(f)=\text{tfp}(\nabla^2f+\Rho f),
\end{equation}
where we write $\text{tfp}$ to indicate the trace-free part of a $\binom02$-tensor
field. It turns out that in this case any solution $D$ is normal, i.e.\ if $D(f)=0$
then $\nabla^{\Cal V}S(f)=0$ and that these solutions are related to conformal
rescalings of $g$ which are Einstein metrics, see \cite{BEG}. 

\medskip

In degrees $1$,\dots, $n-1$, both the splitting operators and the BGG operators are
easier to get. Sections of $\La^kT^*M\otimes\Cal V_1M$ can be viewed as
$\binom1k$-tensor fields, which are skew symmetric in the $k$ lower indices. From \S
\ref{2.7} we conclude that $T: \La^kT^*M\otimes\Cal V_1M\to\La^{k-1}T^*M$ must be a
non-zero multiple of the contraction, while
$\partial:\La^kT^*M\otimes\Cal V_1M\to\La^{k+1}T^*M$, up to a non-zero multiple, must
be given by lowering the upper index and the completely alternating. We also know
that $\Upsilon_k=\ker(T)\cap\ker(\partial)$. For a section
$\psi_1\in\Ga(\Upsilon_k)$, we know that $\partial\ps_1=0$, which together with
Proposition \ref{prop2.6} shows that $d^{\nabla^{\Cal V}}(0,\psi_1,\psi_2)$ has
vanishing first component, while the second component is given by
$d^\nabla\psi_1+\partial\psi_2$. The form $S(\psi_1)$ is characterized by the fact
that this lies in the kernel of $T$, which says that $\psi_2=-T(d^\nabla\psi_1)$. So
up to a non-zero factor, this is obtained by alternating $\nabla\psi_1$ in the lower
indices and then forming the unique contraction.

This also shows that $D(\psi_1)$ is the component of the tracefree part of
$d^{\nabla}\psi_1$ that lies in the kernel of the complete alternation. This can be
expressed as
$$
d^{\nabla}\psi_1-\partial(T(d^\nabla\psi_1))-T(\partial(d^{\nabla}\psi_1)),
$$
so to obtain a more explicit formula, one only has to make the operations $T$ and
$\partial$ explicit.  There is a simplification, however. From the definition of
$d^{\nabla^{\Cal V}}$ it follows that the component of
$d^{\nabla^{\Cal V}}\o d^{\nabla^{\Cal V}}$ that maps $\La^kT^*M\otimes\Cal V_i$ to
$\La^{k+2}T^*M\otimes\Cal V_{k-1}$ is given by
$d^\nabla\o\partial+\partial\o d^\nabla$. However, this component vanishes by
Proposition \ref{prop2.6}, so
$\partial(d^{\nabla}\psi_1)=-d^{\nabla}(\partial\psi_1)=0$. Hence $D(\psi_1)$ is the
tracefree part of $d^\nabla\psi_1$, i.e.\
$$ D(\psi_1)=d^{\nabla}\psi_1-\partial(T(d^\nabla\psi_1)).
$$
While the above description of the splitting operator extends to $k=n-1$,
i.e.\ $S(\psi_1)=(0,\psi_1,-T(d^\nabla\psi_1))$, things are a bit different for the
last BGG operator. The point here is that $T:\La^{n-1}T^*M\to \La^nT^*M\otimes TM$ is
a linear isomorphism, which is reflected in the fact that
$\Upsilon_n\subset\La^nT^*M\otimes\Cal V_2M$ (and since this is a line bundle, we
must have equality). Hence the first two components of $d^{\nabla^{\Cal V}}S(\psi_1)$
already vanish in this case, so it has to have the form $(0,0,D(\psi_1))$. Explicitly
$D(\psi_1)=-d^{\nabla}(T(d^\nabla\psi_1))-\partial^{\Rho}(\psi_1)$.

\subsection{Some general results}\label{2.12}
Several properties observed in \S \ref{2.11} above reflect general features that
occur in all BGG sequences, in particular, this concerns the first BGG operators. As
initially observed in \cite{BCEG}, the order $r$ of the first BGG operator can be
easily read off from the highest weight of the representation $\Bbb V$ that defines
the BGG sequence in question. Moreover, it is shown there, that mapping $\Bbb V$ to
$(\Bbb V_0,r)$ actually gives rise to a bijection between irreducible
representations of $O(b)$ and pairs consisting of an irreducible representation of
$O(n)$ and an integer $r\geq 1$. So an irreducible domain bundle and the order of
the first BGG operator can be chosen arbitrarily and then give rise to a unique BGG
sequence realizing a first operator with these properties. There also is a universal
description of the target bundle of the first BGG operator in representation theory
terms, which loosely can be described as that ``largest irreducible component'' in
$S^r_0T^*M\otimes\Cal V_0M$. It is also known that there is a unique natural
projection from $S^rT^*M\otimes\Cal V_0M$ onto this component and the principal part
of the first BGG operator is obtained by applying this projection to a symmetrized
$r$-fold covariant derivative.

\smallskip

In the form of the discussion in higher degrees in \S \ref{2.11} (i.e.\ without
making the tensorial operations $T$, $\partial$ and $\partial^{\Rho}$ explicit), one
can derive universal formulae for BGG operators of low order. Several results in that
direction (also in higher order) are available in the literature, both in the setting
of parabolic geometries, see e.g.\ \cites{AHS3,CDS,SloSou} and for BGG-like
constructions on domains in $\Bbb R^n$, see e.g.\ \cite{Arnold-Hu}.

The first important observation here is that the order of BGG operators is directly
linked to the location of the spaces $\Upsilon_k$ (respectively their components for
$n$ even and $k=n/2$). We know in general that $\Upsilon_0=\Cal V_0M$ and for the
standard representation, we have observed above that
$\Upsilon_k\subset\La^kT^*M\otimes\Cal V_1M$ for $k=1,\dots,n-1$, while
$\Upsilon_n=\La^nT^*M\otimes\Cal V_2M$. We have also observed there that the first
and last operator in the sequence have order two, while all other operators are of
order one.

In the cases we consider, it turns out that if a component of $\Upsilon_k$ sits in
$\La^kT^*M\otimes\Cal V_iM$ then each component of $\Upsilon_{k+1}$ sits in
$\La^{k+1}T^*M\otimes\Cal V_jM$ with $j\geq i$ and (the relevant component of) the
BGG operator is of order $j-i+1$. Using this, we derive universal formulae for BGG
operators of order $1$ and $2$ in our setting, generalizations to higher order are
straightforward in principle, but lead to intricate formulae quickly. For
simplicity, we ignore the possibility of more than one component in the formulation,
the necessary changes are quite obvious.

\begin{thm}\label{thm2.12}
  Consider the BGG sequence induced by a representation $\Bbb V$ of $O(b)$ and a
  section $\psi\in\Ga(\Upsilon_k)\subset\Om^k(M,\Cal VM)$ and the BGG operator
  $D:\Ga(\Upsilon_k)\to\Ga(\Upsilon_{k+1})$.

  (1) If $D$ has order $1$, then $D(\psi)=d^\nabla\psi-\partial(T(d^\nabla\psi)))$.

  (2) If $D$ has order $2$, then
  $$
    D(\psi)=(\id-\partial\o
    T-T\o\partial)\left(-d^\nabla(T(d^\nabla\psi))-\partial^{\Rho}(\psi)\right). 
  $$
\end{thm}
\begin{proof}
We know that the is an index $i$ such that $\Upsilon_k\subset\La^kT^*M\otimes\Cal
V_iM$ and then for $\psi\in\Ga(\Upsilon_k)$ we get $S(\psi)_j=0$ for $j<i$ and
$S(\psi)_i=\psi$.

(1) By assumption, $\Upsilon_{k+1}\subset\La^{k+1}T^*M\otimes\Cal V_iM$ and the
component $d^{\nabla^{\Cal V}}(S(\psi))_i$ is by definition given by
$d^{\nabla}\psi+\partial(S(\psi)_{i+1})$. As in \S \ref{2.11}, the characterization of
the splitting operator implies that $S(\psi)_{i+1}=-T(d^{\nabla}\psi)$, so
$d^{\nabla^{\Cal V}}(S(\psi))_i=d^\nabla\psi-\partial(T(d^{\nabla}\psi))$ and this
lies in $\ker(T)$ by construction. As in \S \ref{2.11}, Proposition \ref{prop2.6}
implies that $\partial\o d^\nabla=-d^\nabla\o \partial$ and by assumption
$\partial(\psi)=0$. Hence $d^{\nabla^{\Cal V}}(S(\psi))_i$ lies in $\ker(\partial)$
and hence is a section of $\Upsilon_{k+1}$ which coincides with $D(\psi)$ by
definition.

(2) Here $\Upsilon_{k+1}\subset\La^{k+1}T^*M\otimes\Cal V_{i+1}M$ and by definition
\begin{equation}\label{formula}
d^{\nabla^{\Cal V}}(S(\psi))_{i+1}=
d^{\nabla}(S(\psi)_{i+1})+\partial(S(\psi)_{i+2})-\partial^{\Rho}(\psi).
\end{equation}
As in (1), we see that $S(\psi)_{i+1}=-T(d^{\nabla}\psi)$. We also know that
\eqref{formula} has to lie in the kernel of $T$, which easily implies that
$$
S(\psi)_{i+2}=T(d^\nabla( T(d^\nabla(\psi))))+T(\partial^{\Rho}(\psi)), 
$$
whence the right hand side of \eqref{formula} can be written as
$$
(\id-\partial\o T)(-d^\nabla( T(d^\nabla(\ps)))-\partial^{\Rho}(\psi)). 
$$
To obtain $D(\psi)$, this has to be projected to $\ker(\partial)$, so we have to
apply $(\id-T\o \partial)$ to it. Since $\partial\o\partial=0$, this exactly leads to
the claimed formula. 
\end{proof}

Together with the above discussion, part (1) of this theorem implies that all
conformal Killing operators arise as first BGG operators. Fixing an irreducible
representation $\Bbb V_0$ of $O(n)$, one takes the irreducible representation $\Bbb V$
of $O(b)$ corresponding to $(\Bbb V_0,1)$. The corresponding BGG sequence starts with
a first order operators defined on $\Cal V_0M$ and by part (1) of Theorem
\ref{thm2.12} this is simply given by the projection of the covariant derivative to
the subbundle $\Cal H_1^{\Cal V}M$. This subbundle corresponds to the maximal
irreducible component in $\Bbb R^{n*}\otimes\Bbb V_0$. For example, if $\Bbb V_0=\Bbb
R^n\cong\Bbb R^{n*}$, then this is $S^2_0\Bbb R^{n*}$, while for $\Bbb
V_0=S^\ell_0\Bbb R^{n*}$ with $\ell\geq 2$, one obtains $S^{\ell+1}_0\Bbb R^{n*}$, so
these lead to the conformal Killing operator on vector fields and on trace-free
symmetric tensor fields. Likewise, putting $\Bbb V_0=\La^\ell\Bbb R^{n*}$, one
obtains the intersection of the trace-free part with the kernel of the complete
alternation in $\La^\ell\Bbb R^{n*}\otimes\Bbb R^{n*}$. The corresponding operators
are often referred to as conformal Killing-Yano operators on differential forms.

\section{Projective BGG sequences}\label{3}
Apart from conformal geometry, there is a second parabolic geometry that underlies a
Riemannian metric, namely a \textit{projective structure}. Basically, this structure
on $(M,g)$ is defined by the geodesics of the Levi-Civita connection $\nabla$, viewed
as paths (unparametrized curves). Equivalently, it can be described via the class of
all linear connections $\hat\nabla$ on $TM$, which have the same geodesics as
$\nabla$ up to parametrization. Similarly to conformally invariant operators, this
leads to the concept of projectively invariant differential operators, which have the
same expression in terms of all connections in this projective equivalence class.

Examples of BGG sequences coming from projective differential geometry on domains in
$\Bbb R^n$ are very important in applications. In particular, there is a projectively
invariant version of the Riemannian deformation sequence which in this context is
also known as the \textit{Calabi complex} or the \textit{fundamental complex of
  linear elasticity} and this is heavily used in applied mathematics. Concerning
generalizations to the context of Riemannian geometry, there is a drawback of the
projective BGG sequences, however. They are complexes only in a projectively flat
setting, i.e.\ in the case that the projective equivalence class of connections
contains a flat connection. By a classical result of Beltrami, this condition is very
restrictive for Levi-Civita connections, it turns out to be equivalent to constant
sectional curvature. Hence complexes are only obtained on space forms in this
setting, but the sequences are also available and interesting in projectively
non-flat cases.

To cover more cases in which complexes are obtained, we use a more general setting in
this part of the article that we will describe next. Even in this more general
setting, things are strictly parallel to the discussion in \S \ref{2} and in
particular, we use analogous notation to there throughout the discussion.

\subsection{Background and algebraic setup for projective BGG sequences}\label{3.1}
Rather than starting from a Riemannian manifold $(M,g)$ we will start from a pair
$(M,\nabla)$ where $M$ is an orientable manifold of fixed dimension $n$ and $\nabla$
is a torsion-free linear connection on the tangent bundle $TM$. In addition, we
assume that $\nabla$ preserves a volume form and we fix such a form $\nu\in\Om^n(M)$,
so this is parallel for the induced connection on $\La^nT^*M$ and it gives $M$ an
orientation. Analogous to the orthonormal frame bundle discussed in \S \ref{2.1}, we
then have a volume preserving frame bundle $\Cal{SL}M$ with structure group
$SL(n,\Bbb R)$ and $\nabla$ is induced by a principal connection on that
bundle. Hence representations of $SL(n,\Bbb R)$ give rise to associated bundles and
$SL(n,\Bbb R)$-equivariant maps induce bundle maps that are parallel for the
appropriate induced connection. As in \S \ref{2.1} the standard representation
$\Bbb R^n$ of $SL(n,\Bbb R)$ induces $TM$ and the dual representation $\Bbb R^{n*}$
induces $T^*M$. Via tensorial constructions this leads to all types of tensor bundles
(possibly with additional symmetry properties). The adjoint representation induces
the bundle $\frak{sl}(TM)$ of trace-free endomorphisms of $TM$, the curvature of
$\nabla$ can be viewed as an element of $\Om^2(M,\frak{sl}(TM))$, and the analog of
formula \eqref{R-on-W} holds for all induced connections. Note that outside of the
Riemannian setting, one has to carefully distinguish between $TM$ and $T^*M$ and
correspondingly between tensor fields of different types.

\medskip

The algebraic setup for projective BGG sequences starts from the Lie group
$G:=SL(n+1,\Bbb R)$ and we denote the standard basis of $\Bbb R^{n+1}$ as
$e_0,e_1,\dots,e_n$. Then there is an obvious inclusion
$SL(n,\Bbb R)\hookrightarrow G$ as those maps that fix $e_0$ and preserve the
subspace spanned by $e_1,\dots,e_n$. The Lie algebra $\frak g$ of $G$ consists of
all trace-free matrices of size $(n+1)\x (n+1)$ and we write such matrices in a block
form with blocks of size $1$ and $n$ as
\begin{equation}\label{sl(n+1)}
  \begin{pmatrix} a & Z \\ X & A \end{pmatrix} \text{\
    with\ } a\in\Bbb R, X\in\Bbb R^n, Z\in\Bbb R^{n*} \text{\ and\ } a+\tr(A)=0.  
\end{equation}
Of course, the Lie subalgebra of $SL(n,\Bbb R)\subset G$ corresponds
to the matrices with $a=X=Z=0$ (and hence $\tr(A)=0$). Exactly as in \S \ref{2.2},
the block form defines a $|1|$-grading of $\frak g$, with components spanned by $X$,
$(a,A)$ and $Z$, respectively. Also, the bracket defines an extension to $\frak g_0$ of
the standard representation of $\frak{sl}(n,\Bbb R)$ on $\Bbb R^n\cong\frak g_{-1}$
and its dual representation on $\Bbb R^{n*}\cong\frak g_1$. A complement to
$\frak{sl}(n,\Bbb R)$ in $\frak g_0$ is spanned by the element $E$ corresponding to
$a=\frac{n}{n+1}$ and $A=\frac{-1}{n+1}\Bbb I$, which acts as a grading element. This
also shows that $\frak g_0\cong\frak{gl}(n,\Bbb R)$, so the corresponding associated
bundle can be identified with $\frak{gl}(TM)\cong T^*M\otimes TM$.

The block form and hence the $|1|$-grading of $\frak g$ is determined by the
decomposition $\Bbb R^{n+1}=\Bbb R^n\oplus\Bbb R$ with the summands spanned by
$e_1,\dots,e_n$ and by $e_0$, respectively. Since any irreducible representation
$\Bbb V$ of $G$ can be obtained from $\Bbb R^{n+1}$ via tensorial constructions, we
as in \S \ref{2.2} obtain a grading $\Bbb V=\oplus_{j=0}^N\Bbb V_j$, such that $\frak
g_i\cdot \Bbb V_j\subset\Bbb V_{i+j}$. As in \S \ref{2.3}, $\Bbb V$ gives rise to an
associated bundle of $\Cal{SL}M$, which we denote by $\Cal VM=\oplus_{j=0}^N\Cal
V_jM$. We also obtain bundle maps $\bullet$ as there, for which the analogs of equations
\eqref{Leibniz2}--\eqref{equiv-T*} hold. In the interpretations of the latter
equations, one only has to be careful to use the right extensions of representations
of $\frak{sl}(n,\Bbb R)$ to $\frak g_0\cong\frak{gl}(n,\Bbb R)$.

\subsection{Examples}\label{3.2}
(1) For the standard representation $\Bbb V=\Bbb R^{n+1}$ of $G$ we know that $\Bbb
V=\Bbb V_0\oplus\Bbb V_1$ and as representations of $\frak{sl}(n,\Bbb R)$ this equals
$\Bbb R^n\oplus\Bbb R$. Hence sections can be written as pairs $(\eta,f)$ with
$\eta\in \frak X(M)$ and $f\in C^\infty(M,\Bbb R)$. The bundle maps $\bullet:TM\x
\Cal VM\to\Cal VM$ and $\bullet:T^*M\x \Cal VM\to\Cal VM$ are characterized by
$\xi\bullet(\eta,f)=(f\xi,0)$ and $\al\bullet(\eta,f)=(0,\al(\eta))$. Here the
analogs of \eqref{abelian-1} and \eqref{abelian1} are satisfied trivially since both
sides are $0$. To interpret the analogs of equations \eqref{equiv-T} and
\eqref{equiv-T*} on all of $\frak g_0$, one has to take into account that the
grading element $E$ which corresponds to $-\id_{TM}$ acts via $(\eta,f)\mapsto
(\tfrac{-1}{n+1}\eta,\tfrac{n}{n+1}f)$.

For the dual representation $\Bbb V^*$, one similarly obtains $\Bbb V^*=\Bbb
R\oplus\Bbb R^{n*}$, so sections can be written as pairs $(f,\ph)$ with $f\in
C^\infty(M,\Bbb R)$ and $\ph\in\Om^1(M)$. The definition of the dual action readily
leads to $\xi\bullet(f,\ph)=(-\ph(\xi),0)$ and $\al\bullet(f,\ph)=(0,-f\al)$.

(2) Starting from these two basic examples, one can apply tensorial constructions as
in \S \ref{2.4} to pass to general irreducible representations of $G$. Let us discuss
the example $\La^2\Bbb V^*=\Bbb R^{n*}\oplus \La^2\Bbb R^{n*}$ which is relevant for
elasticity. Sections of the corresponding bundle can be written as $(\ph,\ps)$ with
$\ph\in\Om^1(M)$ and $\ps\in\Om^2(M)$, and we can realize $(\ph,0)$ as $(0,\ph)\wedge
(1,0)$ and $(0,\psi_1\wedge\psi_2)$ as $(0,\psi_1)\wedge(0,\psi_2)$. Using this, one
easily verifies that $\xi\bullet(\ph,\psi)=(-\psi(\xi,\_),0)$ and $\al\bullet
(\ph,\psi)=(0,\al\wedge\ph)$.

Likewise, we can consider $S^2\Bbb V^*=\Bbb R\oplus\Bbb R^{n*}\oplus S^2\Bbb R^{n*}$,
so sections can be written as triples $(f,\ph,\Ph)$, where $f\in C^\infty(M,\Bbb R)$,
$\ph\in\Om^1(M)$ and $\Ph$ is a symmetric $\binom02$-tensor field. Via the
realizations $(f,0,0)=(f,0)\odot (1,0)$, $(0,\ph,0)=(0,\ph)\odot (1,0)$ and
$(0,0,\ph_1\odot\ph_2)=(0,\ph_1)\odot (0,\ph_2)$ one easily computes that
$$
\xi\bullet (f,\ph,\Ph)=(-\ph(\xi),-\Ph(\xi,\_),0)\qquad
\al\bullet(f,\ph,\Ph)=(0,-2f\al,-\al\odot\ph). 
$$

(3) Let us finally look at the adjoint representation $\Bbb W:=\frak g$ in a similar
spirit as in the conformal case. We already know that $\Cal
WM=TM\oplus\frak{gl}(TM)\oplus T^*M$ so we write sections as $(\zeta,\Ph,\ph)$ with
$\zeta\in\frak X(M)$, $\ph\in\Om^1(M)$ and $\Ph$ a $\binom11$-tensor field that we
interpret as an endomorphism of $TM$. As in the conformal case, the main step towards
understanding this is to determine the bilinear map $\{\ ,\ \}:\frak
X(M)\x\Om^1(M)\to\frak{gl}(TM)$ induced by the bracket $\frak g_{-1}\x\frak
g_1\to\frak g_0$. To determine this, we again have to compute $[[X,Z],Y]\in\frak
g_{-1}$ for $X,Y\in\frak g_{-1}$ and $Z\in\frak g_1$, which easily implies that
\begin{equation}\label{bracket-sl}
    \{\eta,\ph\}(\xi)=\ph(\xi)\eta-\ph(\eta)\xi.
\end{equation}
In terms of this operation, we then get
$$
\eta\bullet (\zeta,\Ph,\ph)=(-\Ph(\eta),\{\eta,\ph\},0)\qquad
\al\bullet(\zeta,\Ph,\ph)=(0,-\{\zeta,\al\},\al\o\Ph).
$$

\subsection{The twisted de Rham sequence}\label{3.3}
Using the operations we have just introduced, the discussion of curvature in the
projective setting looks formally almost identical to the discussion in the conformal
case. The definition of curvature is exactly as in \eqref{R-def} and the definition
of Ricci-curvature just needs an equivalent reformulation of \eqref{Ricci}. On takes a
local frame $\{\xi_i\}$ for $TM$ and the dual coframe $\si_i$ for $T^*M$ (i.e.\
$\si_i(\xi_j)=\delta_{ij}$ and defines
\begin{equation}\label{Ricci2}
 \Ric(\eta,\zeta)=\textstyle\sum_i\si_i(R(\xi_i,\eta)(\ze)). 
\end{equation}
This is equivalent to \eqref{Ricci} since a local frame $\{\xi_i\}$ it orthonormal
for $g$ if and only if the dual coframe is $\{\xi_i^\flat\}$. For general linear
connections on $TM$, the Ricci-curvature is not symmetric. Symmetry of $\Ric$ is
equivalent to the fact that $\nabla$ preserves a volume form and hence is satisfied
in the cases we consider. The general definition of the projective Schouten tensor
mixes the symmetric and the skew symmetric part of $\Ric$ with different factors, see
Section 3.1 of \cite{BEG}. In the volume preserving case, this boils down to
$\Rho=\tfrac1{n-1}\Ric$. The projective version of the Cotton-York tensor $Y$ is then
again defined by formula \eqref{Y-def}, so $Y=d^{\nabla}\Rho\in\Om^2(M,T^*M)$.

The choice of the definition of $\Rho$ is motivated by the analog of formula
\eqref{Weyl} using the operations associated to $\frak{sl}(n+1,\Bbb R)$, i.e.\
$$
 R(\xi,\eta)(\ze)=W(\xi,\eta)(\ze)+\{\xi,\Rho(\eta)\}(\ze)-\{\eta,\Rho(\xi)\}(\ze).
 $$ Using this to define the \textit{projective Weyl curvature}
 $W\in\Om^2(M,\End(TM))$ one verifies that $W$ has the same symmetries as $R$ and in
 addition any possible contraction of $W$ vanishes (so in particular it has values in
 $\frak{sl}(TM)$). Thus we obtain the decomposition of $R$ into a trace-free part and
 a trace part in the same form as in the conformal case. Also the relation of the
 curvature quantities to projective flatness is parallel to the conformal case (with
 a shift in dimension): For $n\geq 3$, $\nabla$ is projectively flat if and only if
 $W$ vanishes identically and this implies $Y\equiv 0$. If $n=2$, $W$ always vanishes
 identically and $\nabla$ is projectively flat if and only if $Y$ vanishes
 identically. Having all this at hand, we can proceed formally in exactly the same
 way as in the conformal case.

\begin{definition}\label{def3.3}
Consider an irreducible representation $\Bbb V$ of $SL(n+1,\Bbb R)$ as in \S
\ref{3.1}. Then using the operations $\bullet$ from there, we define the
\textit{twisted connection} $\nabla^{\Cal V}$ on $\Cal VM$ by
\begin{equation}
  \label{twisted-def2}
  \nabla^{\Cal V}_\xi s:=\nabla_\xi s+\xi\bullet s-\Rho(\xi)\bullet s. 
\end{equation}
\end{definition}

The proof of Theorem \ref{2.2} only uses the formal properties of the operations and
hence also implies the following.

\begin{thm}\label{thm3.3}
For $\xi,\eta\in\frak X(M)$ and $s\in\Ga(\Cal VM)$, the curvature $R^{\Cal V}$ of
$\nabla^{\Cal V}$ is given by
  $$
  R^{\Cal V}(\xi,\eta)(s)=W(\xi,\eta)\bullet s+Y(\xi,\eta)\bullet s,
  $$
  where $W$ and $Y$ are the (projective) Weyl curvature and the Cotton--York tensor
  of $\nabla$, respectively. In particular, the connection $\nabla^{\Cal V}$ is flat
  if and only if $\nabla$ is projectively flat.
\end{thm}

As before, we get the covariant exterior derivatives $d^\nabla$ and $d^{\nabla^{\Cal
    V}}$ that act on $\Om^*(M,\Cal VM)$, which both raise the form-degree by one. The
relation between the two operations is formally exactly as in Proposition
\ref{prop2.6} but using the operations associated to $\frak{sl}(n+1,\Bbb R)$. Hence
we again get a twisted de Rham sequence associated to each irreducible representation
$\Bbb V$ of $SL(n+1,\Bbb R)$ and this is a complex if and only if $\nabla$ is
projectively flat.

\subsection{Cohomology bundles and BGG construction}\label{3.4}
The definition of the cohomology bundles $\Cal H_k^{\Cal V}M$ as a quotient works
exactly as in the conformal case. The operators $\partial$ on $\Cal VM$-valued forms
are induced by bundle maps $\partial:\La^kT^*M\otimes\Cal
VM\to\La^{k+1}T^*M\otimes\Cal VM$ such that $\partial\o\partial=0$. As in \S
\ref{2.7}, we obtain natural subbundles
$\im(\partial)\subset\ker(\partial)\subset\La^kT^*M\otimes\Cal VM$ and we define
$\Cal H_k^{\Cal V}M$ as the quotient $\ker(\partial)/\im(\partial)$. By construction,
this bundle is induced by the representation $H^k(\frak g_{-1},\Bbb V)$ of $SL(n,\Bbb
R)$.

Both the elementary arguments and the deeper use of representation theory discussed
in \S \ref{2.7} have analogs in the projective case. In particular, Kostant's theorem
implies that for $\frak g=\frak{sl}(n+1,\Bbb R)$ each of the cohomology spaces
$H^k(\frak g_{-1},\Bbb V)$ is an irreducible representation of $\frak g_0$ and, as
before, $H^0(\frak g_{-1},\Bbb V)\cong\Bbb V_0$.

There also is an interpretation via subbundles of $\La^kT^*M\otimes\Cal VM$ and for
Levi-Civita connections they can be described exactly as in \S \ref{2.7} via inner
products. For more general connections, there also is a distinguished natural
subbundle $\Upsilon_k\subset\ker(\partial)\subset \La^kT^*M\otimes\Cal VM$ which
projects isomorphically onto $\Cal H_k^{\Cal V}M$, but one has to use an alternative
description. In fact, one can directly define $\frak g_0$-equivariant maps
$\partial^*:\La^kT^*M\otimes\Cal VM\to\La^{k-1}T^*M\otimes\Cal VM$ for each $k$, the
so-called \textit{Kostant codifferential}. As the notation suggests, they are adjoint
to $\partial$ with respect to an inner product of Lie theoretic origin, so they
satisfy $\partial^*\o\partial^*=0$, and one puts
$\Upsilon_k:=\ker(\partial)\cap\ker(\partial^*)$. The explicit formula for
$\partial^*$ is not needed for our purposes, it can be found (in a more general
setting) in \cite{Kostant} or section \ref{3.3} of \cite{book}. In simple situations,
it is easy to directly find the representations
$\im(\partial^*)\subset\ker(\partial^*)$ in each degree from elementary
representation theory arguments, see \S \ref{3.5} below for examples.

The adjointness between $\partial$ and $\partial^*$ implies that $\partial$
restricts to a linear isomorphism $\im(\partial^*)\to\im(\partial)$ and that
$\im(\partial)$ is complementary to $\ker(\partial^*)$. Now one defines bundle maps
$T:\La^kT^*M\otimes\Cal VM\to\La^{k-1}T^*M\otimes\Cal VM$ as the inverse of $\partial$
on $\im(\partial)$ and as zero on $\ker(\partial^*)$. Doing this, it is clear that
the equalities in \eqref{T-comp} also hold here and we also get
$\Upsilon_k=\ker(\partial)\cap\ker(T)$. The first two equalities in \eqref{T-spaces}
have an analog here, namely that $\La^kT^*M\otimes\Cal VM$ can be written as
$\ker(T)\oplus\im(\partial)$ or as $\im(T)\oplus\ker(\partial)$.

At this point, the rest of the BGG construction can be carried out in the current
setting without changes. The operators $G$ and $S$ can be defined by exactly the same
formula as in \S \ref{2.8} and Proposition \ref{prop2.8} holds. This allows us to
define the BGG operators as in the conformal case. In the projectively flat case, the
relation between the twisted de Rham sequence and the BGG sequences is exactly as in
Theorem \ref{thm2.9}. Without assuming projective flatness, we get the obvious analog
of Theorem \ref{thm2.10}.

\subsection{General results and examples}\label{3.5}
With rather obvious changes, the general results from \S \ref{2.12} extend to the
projective setting. One has a bijection between irreducible representations $\Bbb V$
of $SL(n+1,\Bbb R)$ and pairs $(\Bbb V_0,r)$ of an irreducible representation of
$SL(n,\Bbb R)$ and and integer $r\geq 1$, which equals the order of the first BGG
operator in the sequence determined by $\Bbb V$. The bundle $\Cal H_1^{\Cal V}M$ now
corresponds to the maximal irreducible component in $S^rT^*M\otimes\Cal V_0M$.

With the projective versions of all operations involved, the explicit formula for
BGG operators of order $1$ and $2$ from Theorem \ref{thm2.12} hold in the projective
setting without changes. Parallel to the discussion in \S \ref{2.12}, this shows that
one obtains the Killing operator on $1$-forms and, more generally, on
$\binom0\ell$-tensor fields and the Killing-Yano operators on differential forms as
first BGG operators in this setting.

To discuss some explicit examples, consider the standard representation
$\Bbb V:=\Bbb R^{n+1}$ of $SL(n+1,\Bbb R)$. This decomposes as
$\Bbb V_0\oplus\Bbb V_1=\Bbb R^n\oplus\Bbb R$. Here we get
$\partial:\La^k\Bbb R^{n*}\to\La^{k+1}\Bbb R^{n*}\otimes\Bbb R^n$ for $k=0,\dots,n-1$
which is a non-zero multiple of the inclusion of the trace-part, so it is injective
for all $k$ and bijective for $k=n-1$. This means that $\Upsilon_k$ is the kernel of
the contraction in $\La^kT^*M\otimes TM$ while $\Upsilon_n=\La^nT^*M$ (and
$\Upsilon_0=\Cal V_0M=TM$). Consequently, only the last BGG operator is of second
order here, while all other BGG operators have order one.

Sections of $\Upsilon_k$ for $k=1,\dots,n-1$ can be interpreted as $\binom1k$-tensor
fields $\psi$, which are completely alternating in the lower indices and lie in the
kernel of the unique (up to sign) contraction available in this situation. Since
there are just two components in the decomposition of $\Bbb V$, the splitting
operators can be read off the proof of Theorem \ref{thm2.12}. Up to a non-zero
constant, $T:\La^\ell T^*M\otimes TM\to\La^{\ell-1}T^*M$ is the unique
contraction. Since $T(\psi)=0$ we conclude that $T(d^{\nabla}\psi)$ coincides (up to
that factor) with unique nonzero contraction of $\nabla\psi$, which is the natural
analog of the divergence in this situation. This has to be put into the other
component to obtain $S(\psi)$.

In particular, one gets the usual divergences on vector fields for $k=0$ and the
first BGG operator maps $\eta\in\frak X(M)$ to the trace-free part of the
$\binom11$-tensor field $\nabla\eta$. Explicitly,
$D(\eta)=\nabla\eta-\tfrac1{n}\div(\eta)\id$, where $\div(\eta)$ denotes the
divergence of $\eta$.  There is a similar formula for the next BGG operators, which
is easy to derive. For the last BGG operator, the target is one-dimensional and on
the relevant space both $T$ and $\partial$ are zero, so the universal formula in part
(2) of Theorem \ref{thm2.12} simplifies to a linear combination of $d\div(\psi)$ and
$\partial^{\Rho}(\psi)$ for
$\psi\in\Ga(\Upsilon_{n-1})\subset\Om^{n-1}(M,TM)$. Observe that
$\partial^{\Rho}(\psi)$ up to a multiple is just the alternation of
$(\xi_1,\dots,\xi_n)\mapsto \Rho(\xi_1)(\ps(\xi_2,\dots,\xi_n))$.

\smallskip

For the dual $\Bbb W=\Bbb V^*$ of the standard representation, we get
$\Bbb W=\Bbb W_0\oplus\Bbb W_1=\Bbb R\oplus\Bbb R^{n*}$. The maps
$\partial:\La^k\Bbb R^{n*}\otimes\Bbb R^{n*}\to\La^{k+1}\Bbb R^{n*}$ are non-zero
multiples of the alternation, and hence are bijective for $k=0$ and surjective for
all $k=0,\dots,n-1$. Hence sections of $\Upsilon_0$ are just smooth functions,
$\Upsilon_k\subset \La^k\Bbb R^{n*}\otimes\Bbb R^{n*}$ is the kernel of the complete
alternation for $k=1,\dots,n-1$ and $\Upsilon_n\cong\Bbb R^{n*}$. Hence the first BGG
operator is of second order, while all other BGG operators are of first order. In
degree zero, one gets $S(f)=(f,df)$ while in higher degrees the splitting operators
are just inclusions. The first BGG operator is given by $D(f)=\nabla^2f+\Rho f$, and
this is related to Ricci-flat connections that are projectively equivalent to
$\nabla$, see \cite{BEG}. For $k\geq 1$, sections of $\Upsilon_k$ can be viewed as
$\binom0{k+1}$-tensor fields that are alternating in the first $k$ entries, but lie
in the kernel of the complete alternation. For such a tensor field $\psi$, $D\psi$
then is obtained by forming $d^{\nabla}\psi$, i.e.\ alternating $\nabla\psi$ in the
first $k+1$ entries and then projecting to the kernel of the complete alternation.

\subsection*{Funding and acknowledgments}
This research was funded in whole or in part by the Austrian Science Fund (FWF):
10.55776/P33559. For open access purposes, the author has applied a CC BY public
copyright license to any author-accepted manuscript version arising from this
submission.  This article is based upon work from COST Action CaLISTA CA21109
supported by COST (European Cooperation in Science and
Technology). https://www.cost.eu.
  
This article grew out of a plenary lecture at the Second International Conference on
Differential Geometry in Fez (Morocco). I would like to thank the organizers and in
particular Mohamed Abassi for inviting me to give this lecture. I would also like to
thank Kaibo Hu for stimulating discussions on the topic of this article.

\end{document}